\documentclass[12pt,a4paper]{article}
\usepackage{amsfonts, amssymb, amscd, amsmath, enumerate, verbatim, calc}
\usepackage{graphicx}
\usepackage[colorlinks=true, allcolors=blue]{hyperref} % to make reference live
\usepackage[english]{babel}
\usepackage[letterpaper,top=2cm,bottom=2cm,left=3cm,right=3cm,marginparwidth=1.75cm]{geometry}
\usepackage{leftindex}
\usepackage{setspace}
\usepackage{tikz}
\usepackage{pgfplots}
\usepackage{multicol}
\usetikzlibrary{intersections,positioning,arrows,shapes,trees}
\usepackage{calrsfs}
\usepackage{float}
\usepackage{makeidx}
\usepackage{enumitem}
\usepackage[utf8]{inputenc}
\usepackage{amsthm,mathrsfs,textcomp}
\usepackage[T1]{fontenc}	
\usepackage{mathtools}
\usepackage[super]{nth}
\usepackage{tikz-cd}
\usepackage{placeins}
\usepackage{authblk}
\theoremstyle{plain}

\newtheorem{theorem}{Theorem}[section]
\newtheorem{maintheorem}{Theorem}

\newtheorem{definition}{Definition}[section]
\newtheorem{proposition}[theorem]{Proposition}
\newtheorem{lemma}{Lemma}

\newtheorem{rem}{Remark}[section]

\numberwithin{equation}{section}
\begin{document}

\title{  Circles and line segments as independence attractors of graphs}
\author[1]{Garima Khetawat\footnote{22ma05002@iitbbs.ac.in}}
\author[2]{Moumita Manna \footnote{Corresponding author, s23ma09002@iitbbs.ac.in}}
\author[3]{Tarakanta Nayak	\footnote{tnayak@iitbbs.ac.in}}
\affil[1,2,3]{Department of Mathematics, 
Indian Institute of Technology Bhubaneswar, India}
\date{}
\maketitle
 \linespread{1}
\begin{abstract}
 By an independent set in a simple graph $G$, we mean a set of pairwise non-adjacent vertices in $G$. The independence polynomial of $G$ is defined as $I_G(z)=a_0 + a_1 z + a_2 z^2+\cdots+a_\alpha z^{\alpha}$, where $a_i$ is the number of independent sets in $G$ with cardinality $i$ and $\alpha$ is the cardinality of a largest independent set in $G$,  known as the independence number of $G$.  Let $G^m$ denote the $m$-times lexicographic product of $G$ with itself. The independence attractor  of $G$, denoted by   $\mathcal{A}(G)$, is defined as  $\mathcal{A}(G) =  \lim_{m\rightarrow \infty} \{z: I_{G^m}(z)=0\}$, where the limit is taken with respect to the Hausdorff metric on the space of all compact subsets of the plane. This paper deals with independence attractors that are topologically simple. It is shown that $\mathcal{A}(G)$ can never be a circle.  
If  $\mathcal{A}(G)$ is a line segment  then it is proved that the line segment is $[-\frac{4}{k}, 0]$ for some $k \in \{1, 2, 3, 4 \}$. Examples of graphs with independence number four are provided whose independence attractors are line segments.

\end{abstract}

%\clearpage
 \section{Introduction}
 Let $G$ be a simple graph, i.e., without a loop or a parallel edge with the vertex set $V(G)$ and the edge set $E(G)$. Two vertices $u$ and $v$ of $G$ are said to be adjacent, we denote it by $u\sim v$, if there is an edge between $u$ and $v$. An {\em independent set} in $G$ is a set of pairwise non-adjacent vertices in it. By an {\em $i$-independent set}, we mean an independent set with cardinality $i$. Let $a_i$ be the number of $i$-independent sets in $G$. The {\em independence polynomial} of $G$  is defined as $I_G(z) =a_0 + a_1 z + a_2 z^2+\cdots+a_\alpha z^{\alpha}$, where $\alpha$ is the cardinality of a largest independent set (\cite{GH}). The number $\alpha$ is called the {\em independence number} of the graph $G$. For every graph $G$, $a_0=1$, as the empty set is the only independent set with no element. Similarly, $a_1$ is the number of vertices in $G$ as each vertex constitutes an independent set with cardinality one. Further, $a_i \neq 0$ for all $i=2, 3,\ldots, \alpha$.  By an empty graph, we mean a graph with no edge but with at least one vertex. The independence polynomial of an empty graph with $n$ vertices is $1+nz+\binom{n}{2}z^2+\binom{n}{3}z^3\cdots +\binom{n}{n}z^n$. By a graph, we mean a non-empty graph unless stated otherwise. 
 
\par The roots of an independence polynomial, also called independence roots of a graph, are important. Their nature and location have been studied by many researchers.  Csikvari proved that, for every graph, the independence root with smallest modulus is a negative real number~\cite{PC}. In ~\cite{CS}, it is shown that all the independence roots of a claw-free graph are real.  Considering specific families of graphs, Brown et al. proved that independence roots are dense in the plane. The  independence roots of paths are dense in $(-\infty, -\frac{1}{4}]$~\cite{brownetal2004}.  We are concerned with the limit set of independence roots of a sequence of graphs arising out of  lexicographic product. 
\par
Given two graphs $G$ and $H$, the {\em lexicographic product {\rm(}or composition{\rm )}} of $G$ with $H$, denoted by $G[H]$, is defined as the graph with vertex set $V(G[H])=V(G) \times V(H)$, where $V(G)$ and $ V(H)$ denote the vertex sets of $G$ and $H$ respectively. Two vertices $(u,v)$ and $(u^{\prime}, v^{\prime})$  in $V(G[H])$ are adjacent in $G[H]$ if and only if either $u\sim u^{\prime}$ in $ G$, or $u=u^{\prime}$ and $v\sim v^{\prime}$ in $H$. This amounts to replacing each vertex of $G$ with a copy of $H$ and joining two vertices of two different copies according to the adjacency of vertices of $G$.  The $m$-times lexicographic product of a graph $G$ with itself is denoted by $G^m$. 
For a polynomial $f$, let $\text{Roots}(f)$ denote the set of all its roots, i.e.,  $\text{Roots}(f)=\{z: f(z)=0\}$.
The set $ \text{Roots} (I_{G^m})$ is finite and therefore a compact subset of the plane for each $m$. The limit  of $ \text{Roots} (I_{G^m})$ exists as $m \to \infty$ with respect to the Hausdorff metric (see Section \ref{ind-attractor-fractal}), and is called the \textit{independence attractor}  $\mathcal{A}(G)$ of  the graph $G$~\cite{hickman2001}.
A key fact relating the lexicographic product of a graph with itself and its independence polynomial is given by  $I_{G^2}(z) =I_G (I_G (z)-1)$ (see Theorem 1.1,~\cite{brownetal2003}). Denoting the reduced independence polynomial $I_G (z)-1$ of $G$ by $P_G (z)$, note that  $I_{G^2}(z) =P_G (P_G (z))+1=P_G ^2 (z)+1$, and in general, $I_{G^m}(z) =P_G ^m (z)+1$. In other words,
\begin{equation}
	 Roots(I_{G^m})  =\{z: P_G^m (z)=-1\}~\mbox{for all}~ m.
\end{equation}   
By virtue of this equality, the iterative behaviour of $P_G$ assumes significance for investigating the independence attractor of $G $. This simple but very important fact  gives rise to the following equivalent definition of the independence attractor.
\begin{definition}$\mathcal{A}(G) =  \lim_{m\rightarrow \infty} P_G ^{-m}(-1),~\mbox{where }~ P_G ^{-m}(-1) =\{z: P_G^m (z)=-1\}$.
\end{definition}
The limit of $P_G ^{-m}(-1)$ as $m \to \infty$, usually referred to as independence fractal and denoted by $\mathcal{F}(G)$, has been a useful tool to understand the independence attractor  of $G$. In fact,  these two objects are the same except when $-1$ is a multiple root of $I_G$ for every non-empty graph $G$.  If $-1$ is a multiple root of $I_G $ then  $\mathcal{A}(G)$ is a disjoint union of $\mathcal{F}(G)$ and $\cup_{m \geq 1} P_G ^{-m}(-1)$ (see Theorem~\ref{Thm_attr1}). The Julia set of a polynomial $Q$, denoted by  $\mathcal{J}(Q)$, is defined as the boundary of its filled Julia set $K(Q)=\{z: \{Q^n(z)\}_{n>0} ~\mbox{is bounded}\}$. It is always a closed set.  A very useful result obtained by Brown et al.  states that $\mathcal{F}(G) =\mathcal{J}(P_G)$ ~\cite{brownetal2003} and this is the basis of the proofs of our main results. We revisit these ideas and reproduce the proofs, often with additional explanations for the sake of completeness. Unlike ~\cite{brownetal2003} and \cite{tk-barik-20} where independence fractal is the primary object of study, this article considers independence attractor as the main object and uses the idea of independence fractal to understand the former.
\par
For an empty graph (i.e., without any edge) $G$ on $n$ vertices,  $P_G (z)= (1+z)^n -1$ and therefore, $\mathcal{A}(G) =\{-1\}$. For a complete graph $K_n$ on $n$ vertices, $P_{K_n} (z)=nz$ and $P^{-m}_{K_n}(-1)=-\frac{1}{n^m}$. Thus, $\mathcal{A}(K_n)=\{0\}$. This article deals with the independence attractors of graphs that are neither  empty nor complete. Independence attractor of graphs with independence number two can be totally disconnected or connected (\cite{hickman2001}).

This article investigates the possibility of some topologically simple independence attractors, namely a circle or a line segment. The first possibility is ruled out.  
\begin{maintheorem}
The  independence attractor of any graph cannot be a circle.
\label{nocircle}
\end{maintheorem}

 Barik et al. in \cite{tk-barik-20}  
proved that, if the  independence attractor of a graph with independence number three is a line segment then it is $[-\frac{4}{k},0]$ for some  $k \in \{1,2,3,4\}$.  We prove this for graphs with all possible independence numbers.
\begin{maintheorem}
	\label{linesegment}
	If a line segment $J$ is the independence attractor of a graph then   $J= [-\frac{4}{k}, 0]$ for some $k \in \{1, 2, 3, 4\}$. 
\end{maintheorem}
For each $k \in \{1,2,3,4\}, $ there are graphs with independence number three whose independence attractor is $[-\frac{4}{k},0]$.  Indeed, the following are proved in \cite{tk-barik-20}, the exact references of which as appeared in \cite{tk-barik-20}  are given in brackets: (i)There are exactly $17$ connected graphs whose independence attractor is $[-4,0]$ (Theorem 3.4); (ii) An algorithm is given to find all graphs with $[-2,0]$ as independence attractor and it is seen that there are at least $902$ such graphs (page 70); (iii) There is a connected graph with $[-\frac{4}{3},0]$ as independence attractor (Theorem 3.3 and Figure (3)) and (iv)  $\mathcal{A}(K_4 \cup K_4 \cup K_1)= [-1,0]$ (Theorem 4.1).  We provide examples of graphs with independence number four whose independence attractor is $[-\frac{4}{k},0]$ for each $k=1,2,3,4$.
\begin{maintheorem}
	For each $k\in \{1,2,3,4\}$, there is at least one graph with independence number four such that its independence attractor is $[-\frac{4}{k}, 0]$. Further, there is no disconnnected graph whose independence attractor is $[-\frac{4}{k}, 0]$ for any $k=1,2$. 
	\label{exampls-ind-no-4}
	\end{maintheorem}

 We have also shown  that $\mathcal{F}(G) =\mathcal{A}(G) $ whenever $\mathcal{A}(G)$ is a line segment (see Theorem~\ref{attractor=fractal}). 
 \par  
The structure of this article follows. Section \ref{ind-attractor-fractal} presents the connections between the independence attractor, independence fractal and the Julia set of the reduced independence polynomial of a graph. Theorem~\ref{nocircle} and Theorem~\ref{linesegment} are proved in Section \ref{circles-linesegments}. Section \ref{examples} contains the proof of Theorem~\ref{exampls-ind-no-4}.
\section{Independence attractors and fractals}
\label{ind-attractor-fractal}
The connections between the independence fractal, independence attractor and the Julia set of the reduced independence polynomial of a graph are already reported in \cite{brownetal2003} and  \cite{hickman2001}. In this section, we revisit these results, and  their proofs are reproduced with additional explanations.

To proceed, we begin with few definitions and results that can be found in   \cite{beardon}. A point $z$ is a periodic point of $f$ if for some positive integer $k,~ f^k(z) = z$, where the smallest such $k$ is called the period of $z$. If $k=1$, then $z$ is called a fixed point of $f$. A finite fixed point of $f$  is called attracting if $0\leq |f'(z)|<1$ (super-attracting if $f'(z)=0$); repelling if $|f'(z)|>1$; rationally indifferent if $f'(z)$ is a root of unity and irrationally indifferent if $|f'(z)|=1$, but $f'(z)$ is not a root of unity. Repelling and rationally indifferent fixed points lie in the Julia set while an attracting (in particular, a super-attracting) fixed point lies in the complement of the Julia set, which is commonly referred to as the Fatou set of $f$ and is denoted by $F(f)$. An irrationally indifferent fixed point may lie either in $F(f)$ or in $\mathcal{J}(f)$.
\par The Fatou set is open by definition. A maximally connected  subset of $F(f)$ is called a Fatou component.
A Fatou component $F_0$ of a polynomial $f$ is called $p-$periodic if $p$ is the smallest natural number such that $f^p (F_0) \subseteq F_0$. A $p-$periodic Fatou component is an attracting or parabolic domain if it contains an attracting or rationally indifferent $p-$periodic point in it or on its boundary respectively. An attracting domain is called super-attracting domain if the corresponding periodic point is super-attracting. The Fatou component $F_0$ is called  a Siegel disk if $f: F_0 \rightarrow F_0$ is analytically conjugate to a Euclidean rotation of the unit disk onto itself. It is a well-known fact that every periodic Fatou component of a polynomial is either an attracting domain, a parabolic domain or a Siegel disk.  
We require some useful results on the Fatou set of a polynomial.

\begin{theorem}[Theorem 3.2.5, \cite{beardon}]
\label{poly_map}
For every polynomial  with degree at least two, the point $\infty$ is a super-attracting fixed point and lies in its Fatou set. Consequently, the Julia set of every polynomial is bounded.
\end{theorem}

Each point in the Fatou set of a polynomial ultimately converges to a cycle of periodic points (attracting or rationally indifferent) or lands on a Siegel disk. We restate this well-known fact in a way suitable for our purpose. The forward orbit $\{z, f(z), f^2(z), \ldots\}$ of a point $z$ under $f$  is denoted by $O^+(z)$. 
For a point $z$, let $\omega(z)$ denote the set of all limit points of $O^{+}(z)$. 
\begin{theorem}[Theorem 7.1.2, \cite{beardon} ]
	If \(f\) is a polynomial with degree at least two, and \(z_0\in F(f)\), then $\omega(z_0)$ is a cycle of  periodic points (attracting or rationally indifferent), or is a periodic cycle of  Jordan curves contained in a cycle of Siegel disks.
	\label{Thm3}
\end{theorem}
 There are sets whose backward iterated images converge to the Julia set.
\begin{theorem}[Theorem 4.2.8, \cite{beardon}]
	  Let \(f\) be a polynomial with degree at least two and \(E\) be a compact subset of \(\mathbb{\widehat{C}}\) such that for any \(z \in F(f)\), \(\omega(z) \cap E =\emptyset\). Then for each open set \(U\) containing \(\mathcal{J}(f)\), \(f^{-m}(E)\subseteq U\)  for all sufficiently large \(m\).
	  \label{Thm1}
\end{theorem}
The backward orbit of $w$ is the set $O^-(w):=\{z: ~f^n(z) = w~\text{ for some } n\geq 0\} = \bigcup_{n\geq 0}f^{-n}(w)$. A point $w$ is said to be exceptional for $f$ if $O^{-}(w) \cup O^{+}(w)$  is finite. The following is usually referred to as the expansive property of the Julia set.
\begin{theorem}[Theorem 6.9.4, \cite{beardon}]
    Let \(f\) be a polynomial with degree at least two.
  If \(W\) is a domain that meets \(\mathcal{J}(f)\) and \(K\) is a compact set not containing any exceptional point of \(f\) then for all sufficiently large \(m\), \(f^{m}(W)\supset K\). In particular, this is true for $K=\mathcal{J}(f).$
   	\label{Juliaset-expanding}
 \end{theorem}
For a subset $A$ of the plane, $Cl(A)$ denotes the closure of $A$ in the plane.
    \begin{theorem}[Theorem 4.2.7, \cite{beardon}]
        For a polynomial \(f\) with  degree at least two, the Julia set $\mathcal{J}(f)$ is completely invariant. Further,
        \begin{enumerate}[label=(\roman{*})]
            \item if \(z_0\) is not an exceptional point, then \(\mathcal{J}(f) \subseteq Cl(O^-(z_0))\);
            \item if \(z_0 \in \mathcal{J}(f)\), then \(\mathcal{J}(f) = Cl(O^-(z_0))\).
        \end{enumerate}
        \label{Juliaset}
    \end{theorem}

For $\epsilon >0$, an  $\epsilon$-neighbourhood of a set $A \subset \mathbb{C}$, denoted by  $[A]_{\epsilon}$, is defined as the set $\bigcup_{a\in A}\{z \in \mathbb{C}:~|z-a| < \epsilon\}$.  The Hausdorff distance $d_H (A,B)$ between two compact subsets $A$, $B$ of $\mathbb{C}$ is given by $d_H(A,B) = \text{inf}\{\epsilon: A\subset [B]_{\epsilon} \text{ and } B\subset [A]_{\epsilon}\}$ ( see Chapter 7, \cite{munkres}). Note that $[A]_{\epsilon_1} \subset [A]_{\epsilon_2}$ for all $\epsilon_1  < \epsilon_2$. Recall that, for a point $z_0$, $f^{-m}(z_0) =\{z: f^m (z)=z_0\}$. This is a compact subset of the plane. For certain points $z_0$, the set $f^{-m}(z_0)$ converges to the Julia set of $f$.
 
\begin{theorem}[Appendix A, \cite{brownetal2003}]
        Let \(f\) be a polynomial with degree at least two. If \(z_0\) is neither an attracting periodic point nor contained in any Siegel disk of $f$, then $ \lim_{m\rightarrow \infty}f^{-m}(z_0) = \mathcal{J}(f),$ where the limit is taken with respect to the Hausdorff metric.
        \label{Thm5}
    \end{theorem}

 \begin{proof}
Let \(f\) be a polynomial and \(z_0\) be as given in the statement. We first show that  $ \lim_{m\rightarrow \infty}f^{-m}(z_0) = \mathcal{J}(f) $  is equivalent to the following statements: For every $\epsilon>0$,
        \begin{enumerate}[label=(\roman{*})]
            \item \(f^{-m}(z_0) \subseteq [\mathcal{J}(f)]_{\epsilon}\), and
            \item \(\mathcal{J}(f)\subseteq [f^{-m}(z_0)]_{\epsilon}\) for all sufficiently large \(m\).
        \end{enumerate}
        
       If \(\lim_{m\rightarrow \infty}f^{-m}(z_0) = \mathcal{J}(f)\) then for every \(\epsilon >0\),  \(d_H(f^{-m}(z_0),~\mathcal{J}(f)) < \epsilon\) for all sufficiently large $m$. For each such $m$, define \(A_m := \{r: f^{-m}(z_0)\subseteq [\mathcal{J}(f)]_r \text{ and } \mathcal{J}(f) \subseteq [f^{-m}(z_0)]_r\}\).  Then the infimum of the set \(A_m\), i.e., $d_H(f^{-m}(z_0),~\mathcal{J}(f)) $ is less than $\epsilon$ and so the set \(A_m\) has an element \(\tilde{\epsilon} \leq \epsilon\). By the definition of \(A_m\), \(f^{-m}(z_0)\subseteq [\mathcal{J}(f)]_{\tilde{\epsilon}} \text{ and } \mathcal{J}(f) \subseteq [f^{-m}(z_0)]_{\tilde{\epsilon}}\). Since \(\tilde{\epsilon} \leq \epsilon\), statements \ (i)  and  (ii)  follow.
        \par
        Conversely, let $\epsilon >0$. Then for  all sufficiently large $m$, \(f^{-m}(z_0) \subseteq [\mathcal{J}(f)]_{\frac{\epsilon}{2}}\) and \(\mathcal{J}(f)\subseteq [f^{-m}(z_0)]_{\frac{\epsilon}{2}}\),  i.e., $\frac{\epsilon}{2}\in A_m$  for all sufficiently large $m$. This implies that the inf $A_m \leq \frac{\epsilon}{2} < \epsilon$. In other words, $d_H(f^{-m}(z_0),~\mathcal{J}(f)) < \epsilon$   for all sufficiently large $m$. Consequently, $\lim_{m\rightarrow \infty}f^{-m}(z_0) = \mathcal{J}(f)$.
        
        Now in order to prove the statements  (i)  and  (ii), let $\epsilon >0$.
        
\underline{Proof of (i):} If \(z_0 \in \mathcal{J}(f)\) then the complete invariance of  the Julia set (see Theorem~\ref{Juliaset}) gives that
        $f^{-m}(z_0) \subseteq \mathcal{J}(f) \subset [\mathcal{J}(f)]_{\epsilon}, \text{ for all } m$.  Now, let \(z_0 \in F(f)\) be such that \(z_0\) is neither an attracting periodic point and nor contained in any Siegel disk. By Theorem \ref{Thm3}, for every $z$ in the Fatou set of $f$, \(\omega(z)\) is disjoint from $\{z_0\}$.  Since  \([\mathcal{J}(f)]_{\epsilon}\) is an open set containing \(\mathcal{J}(f)\), it follows from Theorem \ref{Thm1} that 
        \(f^{-m}(z_0) \subset [\mathcal{J}(f)]_\epsilon\) for all sufficiently large \(m\).

\underline{Proof of (ii):}  Since \(\mathcal{J}(f)\) is compact by Theorem~\ref{poly_map}, there are finitely many disks $B_i, i=1,2,,3,\cdots, k$, each with radius $\frac{\epsilon}{2}$ and center at some point of the Julia set such that $\mathcal{J}(f) \subset \cup_{i=1}^{k} B_i $. An exceptional point is either a  fixed  or a  two periodic super-attracting point (see Theorem 4.1.2,~\cite{beardon}). By the hypothesis of this theorem,  \(z_0\) is not  exceptional. Since each $B_i$ intersects the Julia set, it follows from Theorem \ref{Juliaset-expanding} that for all sufficiently large \(m\), $ z_0 \in  f^{m}(B_i) $, i.e., $B_i \cap f^{-m}(z_0)  \neq \emptyset$ for each $i$. Thus, there is an $N$ such that $B_i \cap f^{-m}(z_0)  \neq \emptyset$ for all $m>N$ and all $i$.
      Hence, for all   \(m > N\),
       $\bigcup_{i=1}^{k} B_i \subset [f^{-m}(z_0)]_{\epsilon}$. We are done since the Julia set is contained in  $\bigcup_{i=1}^{k} B_i$.
    \end{proof}
The point $0$ is always a repelling fixed point of the reduced independence polynomial of every graph. This along with Theorem~\ref{Thm5}  gives rise to the following useful result. 

\begin{theorem} [Theorem 3.3, \cite{brownetal2003}]
	For every graph $G$ different from $K_1$,  $\mathcal{F}(G) =\mathcal{J}(P_G)$.
	\label{indfractal=Juliaset}
\end{theorem} Note that $P_{K_1}(z)=z$ and $\mathcal{J}(P_{K_1})=\widehat{\mathbb{C}}$ whereas $\mathcal{F}(K_1) =\{0\}$.
We require two results to prove the relation between $\mathcal{F}(G)$ and $\mathcal{A}(G)$.
 \begin{theorem}(Rational Root Theorem)
        Let $f(z)=a_0 + a_1z +a_2z^2 + \cdots + a_nz^n$ be a polynomial with integer coefficients. If $a_0, a_n \neq 0$ then for each rational root $\frac{p}{q}$ of $f$ with $\gcd(p,q)=1$,  $p$ divides $a_0$ and $q$ divides $a_n$.
\end{theorem}

\begin{lemma}
     Let $f$ be a polynomial with degree at least two and with positive integer coefficients. If there is a Siegel disk $D$ of $f$ then $D$ is disjoint from $f^{-m}(\mathbb{Z}[i])$ for each non-negative integer  $m$ where $\mathbb{Z}[i] =\{a+ib: a,b \in \mathbb{Z}\}$.
    \label{Siegel}
\end{lemma}
\begin{proof}
If there is a point $z \in D \cap f^{-k}( \mathbb{Z}[i])$ for some non-negative integer $k$ then $f^m (z) \in \mathbb{Z}[i]$ for all $ m \geq k$. Since 	$\mathbb{Z}[i]$ is forward invariant under $f$, the forward orbit of $z$ cannot be dense in any Jordan curve. However, this must be the case as $z$ is in a Siegel disk.
\end{proof}
Let $f$ be an analytic function. We say that a point $z_0$ is a root of $f$ with multiplicity $k$ if there is an analytic function $g$ such that 
$g(z_0) \neq 0$ and $f(z) = (z-z_0)^k g(z)$. It is a simple root of $f$ if $k=1$ and a multiple root if $k>1$. Further, if the multiplicity of $z_0$ as a root of $f$ is $k$, then $f(z_0)=0, f'(z_0)=0, \ldots, f^{(k-1)}(z_0)=0$ and $f^{(k)}(z_0)\neq 0$.

\begin{theorem}[Theorem 3.2.14, \cite{hickman2001}]
For every non-empty graph \(G\), \(\mathcal{F}(G) \subseteq \mathcal{A}(G)\). More precisely, the following are true.
\begin{enumerate}
\item If the point $-1$ is not a root of $I_G$, then \(\mathcal{A}(G)=\mathcal{F}(G)\).
\item If the point $-1$ is a simple root of $I_G$, then $\mathcal{A}(G)=\mathcal{F}(G)$. Further, $ \mathcal{F}(G)= Cl~\left(\bigcup_{m\geq 1}P_G^{-m}(-1)\right)$.
\item If the point $-1$ is a multiple root of $I_G$, then $\mathcal{A}(G) $ is a disjoint union of $\mathcal{F}(G) $ and $ (\cup_{m\geq 1}P_G^{-m}(-1))  $. Further, $\mathcal{F}(G)$ is the limit set of $\cup_{m\geq 1}P_G^{-m}(-1)$.
\end{enumerate}
\label{Thm_attr1}
    \end{theorem}

    \begin{proof}
        If \(G\) has independence number one, then \(G=K_n\)  for some \(n\geq 2\) and \(G^m = K_{n^m}\) for all \(m\geq 1\). Hence \(I_{G^m}(z) = 1 + n^mz\). For each \(m\), the only root of \(I_{G^m}(z)\) is \(-\frac{1}{n^m}\). Thus, $ \lim_{m \rightarrow \infty} \text{Roots }(I_{G_m}) = \{0\}$. This gives $\mathcal{A}(G) = \{0\}$.
        Further, \(P_G(z) = nz\). The point \(z=0\) is a repelling fixed point of \(P_G\) and therefore it lies in \(\mathcal{J}(P_G)\). For \(z\neq 0\), the forward orbit of \(z\) converges to \(\infty\). Therefore, \(\mathcal{J}(P_G) = \{0\}\). By Theorem \ref{indfractal=Juliaset}, \(\mathcal{F}(G) = \{0\}\). This implies \(\mathcal{A}(G) = \mathcal{F}(G)\) whenever the independence number of $G$ is one.
        \par
        Let \(G\) be a graph with independence number at least two.
        
           \begin{enumerate}
           	\item If \(-1\) is not a root of \(I_G(z)\), then it is not a fixed point of \(P_G(z)\). If  \(P_G^{p}(-1) = -1\) for some \(p\geq 2\) then \(P_G(P_G^{p-1}(-1)) + 1 = 0\). Since \(P_G^{p-1}(-1) \in \mathbb{Z}\) is a rational root of $P_G (z)+1=0$, by the Rational Root Theorem, $P_G^{p-1}(-1)$ divides the constant term of  \(P_G(z)+1 \), i.e., $1$. This leads to only two possibilities: \(P_G^{p-1}(-1) = 1\) or \(-1\). But \(P_G(1)+1 \neq 0\) and, we have  \(P_G^{p-1}(-1) = -1\). Repeating the argument, it can be eventually  found that \(P_G(-1)=-1\),  which is a contradiction to our initial assumption. In particular, \(-1\) does not lie in any attracting cycle of \(P_G\). Also, by Lemma \ref{Siegel}, \(-1\) is not in any Siegel disk of \(P_G\). Therefore, by  Theorem \ref{Thm5}, $\mathcal{J}(P_G) =\lim_{m \to \infty} P_G ^{-m}(-1)$.   Theorem \ref{indfractal=Juliaset} gives that \(\mathcal{A}(G)=\mathcal{F}(G)\).

            \item  If \(-1\) is a simple root of \(I_G(z)\), then it is a fixed point of \(P_G(z)\) and \(P_G'(-1) \neq 0\). Since \(P_G'\) has integer coefficients, \(P_G'(-1) \in \mathbb{Z}\) and \(|P_G'(-1)| \geq 1\). This implies that \(-1\) is a repelling or a rationally indifferent fixed point of \(P_G\) and therefore lies in \(\mathcal{J}(P_G)\).  By Theorem \ref{Thm5}, $\mathcal{J}(P_G) =\lim_{m \to \infty} P_G ^{-m}(-1)$. Now it follows from  Theorem \ref{indfractal=Juliaset} that   \(\mathcal{A}(G) =\mathcal{F}(G) \).
            \par
            As \(-1 \in \mathcal{J}(P_G)\), $\mathcal{J}(P_G) = Cl~\left(\bigcup_{m\geq 1}P_G^{-m}(-1)\right)$  by Theorem 2.5 (ii). Hence, \(\mathcal{F}(G) = Cl~\left(\bigcup_{m\geq 1}P_G^{-m}(-1)\right).\)

            \item If \(-1\) is a multiple root of \(I_G(z)\), then it is a super-attracting fixed point of \(P_G(z)\). To show that \(\mathcal{A}(G) = Cl~\left(\bigcup_{m\geq 1}P_G^{-m}(-1)\right)\),  let \(\epsilon > 0\). Then it is equivalent to prove the following two statements, for all sufficiently large \(n\):
            \begin{enumerate}[label=(\alph{*})]
                \item \(P_G^{-n}(-1) \subseteq [Cl~\left(\bigcup_{m\geq 1}P_G^{-m}(-1)\right)]_ \epsilon\), and 
                \item \(Cl~\left(\bigcup_{m\geq 1}P_G^{-m}(-1)\right) \subseteq [P_G^{-n}(-1)]_ \epsilon\).
            \end{enumerate}
            \par
            Part (a) is obvious as \(P_G^{-n}(-1) \subseteq \bigcup_{m\geq 1}P_G^{-m}(-1) \) for all \(n>0\). For showing part (b), note that  \(P_G^{-(m-1)}(-1) \subseteq P_G^{-m}(-1)\) for all \(m\geq 2\). Consider \(  B_m = [P_G^{-m}(-1)]_\epsilon \). Then $\bigcup_{m \geq 1}P_G^{-m}(-1) \subset \bigcup_{m 
            \geq 1} B_m$.  If \(z\) is a limit point of \(\bigcup_{m\geq 1}P_G^{-m}(-1)\) then the open ball \(B(z, \frac{\epsilon}{2})\) contains a point \(z'\) such that \(z' \in P_G^{-k}(-1)\) for some natural number \(k\). Therefore, \(z \in [P_G^{-k}(-1)]_\epsilon \subseteq \bigcup_{m \geq 1}B_m\). This gives that $\{B_m : m \geq 1\}$ is an open cover of    \(Cl\left(\bigcup_{m\geq 1}P_G^{-m}(-1)\right) \), which is clearly a compact subset of the plane.  Thus, there are finitely many $B_m$'s whose union contains \(Cl\left(\bigcup_{m\geq 1}P_G^{-m}(-1)\right) \). Since $B_m  \subset B_{m+1}$ for each $m$,  we have
            $ Cl\left(\bigcup_{m\geq 1}P_G^{-m}(-1)\right) \subseteq B_n, \text{ for all sufficiently large} ~ n$.
            \par It remains to be shown that $\mathcal{F}(G)$ is the limit set of  $\bigcup_{m\geq 1}P_G^{-m}(-1)$. If $z \neq -1$ and $P_G (z)=-1$ then it follows from Theorem~\ref{Thm5} that $\mathcal{J} (P_G) =\lim_{m \to \infty} P_G ^{-m}(z)$. This gives that the limit set of $\bigcup_{m\geq 1}P_G^{-m}(-1)$ is exactly equal to  $\mathcal{J} (P_G)$. Clearly, $\mathcal{J} (P_G)$ is disjoint from $\bigcup_{m\geq 1}P_G^{-m}(-1)$.
            \end{enumerate}
    \end{proof}

\section{Circles and line segments}
\label{circles-linesegments}
The following lemma is to be applied to reduced independence polynomials.
\begin{lemma}
    Let \(P\) be a polynomial with degree at least two and with positive coefficients such that $P(0) = 0$ and $P'(0) > 1$. Then its Julia set is symmetric with respect to the real line, contains the point \(0\) and does not intersect the positive real line.
     \label{real_points}
 \end{lemma}
 \begin{proof}
      Since all the coefficients of \(P\) are real, $P(\overline{z})=\overline{P(z)}$, and in general,  $P^n(\overline{z})=\overline{P^n(z)}$ for all $n$ and $z \in \mathbb{C}$.
	Therefore $\{P^n (z)\}_{n>0}$ is bounded if and only if   $\{P^n (\overline{z})\}_{n>0}$ is bounded. Consequently, $z\in \mathcal{J}(P)$ if and only if  $\overline{z}\in \mathcal{J}(P) $. 
	\par The point $0$ is a repelling fixed point of $P$, and therefore it is in the Julia set of $P$. 
 \par 
For every $x>0$, we have $P(x)>x$ and $P$ is strictly increasing on the positive real line. This implies that \(P^2(x)>P(x)>x\), and the sequence $\{P^{n}(x)\}_{n>0}$ is strictly increasing. If this sequence converges to a real number (it is so whenever it is bounded) then the limit would have to be a positive fixed point of $P$. However, $P$ has no positive fixed point. Thus, $P^n (x) \to \infty$ as $n \to \infty$ for all $x>0$. 
Since $\infty$ is a super-attracting fixed point of $P$, every positive number is in the attracting domain corresponding to $\infty$. In other words, $(0, \infty)$ is contained in the Fatou set of $P$, i.e., the Julia set of $P$ does not intersect the positive real line. 
 \end{proof}
%
% \textbf{ By \(K_m\), we mean a complete graph on \(m\) vertices and \( K_m ^c\) represents the empty graph on \(m\) vertices. \(P_m\) is the path graph on \(m\) vertices. We consider graphs with at least two vertices. The reduced independence polynomial has no constant term and all other coefficients are natural numbers. }

\subsection{No circle}
Complement of a graph $G$, denoted by $G^c$ is the graph with the same vertex set as that of $G$ such that two vertices in $G^c$ are adjacent if and only if they are not adjacent in $G$. As observed by
Brown et al. in \cite{brownetal2003},  the independence fractal of the graph $K_n ^c$ is $\{0\}$ when $n=1$ and it is the circle $\{z\in \mathbb{C}: |z+1| = 1\}$ for all $n\geq2$. We now prove that  $ K_n^c$ is the only graph whose independence fractal is a circle. This is required for the proof of Theorem~\ref{nocircle}.

\begin{lemma}
  For a natural number $n \geq 2$, if $G$ is a graph with independence number $n$ whose independence fractal is a circle then $G = K_n ^c$ and the circle is $\{z\in \mathbb{C}: |z+1| = 1\}$.
  \label{circle}
\end{lemma}

\begin{proof}
    Let $P(z)=a_1 z +a_2 z^2+\cdots +a_n z^n$ be the reduced independence polynomial of $G$ and its Julia set be a circle. By Lemma \ref{real_points}, $ \mathcal{J}(P) $ is a circle containing $0$ and is symmetric with respect to the real line. Further, this circle does not contain any positive real number. Therefore, $\mathcal{J}(P)=\{z\in\mathbb{C}:|z - z_0|=|z_0|\}$ for some negative real number $z_0$. Let \(C\) be the unit circle and define \(\phi(z) = \frac{z -z_0}{-z_0}\). Then \(\phi (\mathcal{J}(P))=C\) and \(C\) is completely invariant under the map \(\phi P \phi^{-1}\), which is a polynomial with degree \(n\). By Theorem 1.3.1 in \cite{beardon}, \(\phi P \phi^{-1}(z) =  \beta z^n\), for some \(\beta \in \mathbb{C}\) with \(|\alpha|=1\). Thus, 
\begin{equation}
 \phi (P(z)) = \beta (\phi (z))^n ~\mbox{ for all}~  z \in \mathbb{C} .
 \label{circle-conjugacy}
 \end{equation}
  Putting \(z = 0\) in Equation (\ref{circle-conjugacy}), we get \(\beta = 1\). Differentiating Equation (\ref{circle-conjugacy}) and putting $z=0$, we have $a_1 =n$. Therefore, there is only one independent set with cardinality $n$, i.e., $a_n=1$. Now comparing the coefficients of $z^n$ on both sides of Equation~(\ref{circle-conjugacy}), we get $z_0 =-1$. Thus, $P(z) = (1+z)^n -1$. Therefore $G= K_n ^c$ and the Julia set of $P$ is the circle $\{z\in \mathbb{C}: |z+1| = 1\}$.
\end{proof}

\begin{rem}
	\begin{enumerate}
		\item If $G=K_n ^c, n >1$ then $G^m =K_{n^m}^c$,  $I_{G^m}(z)=(1+z)^{n^m}$ 	and $P_{G}^{m}(z)=(1+z)^{n^m} -1$, which gives that $\mathcal{A}(G)=\{-1\}$. It  follows from  Lemma~\ref{circle} that $\mathcal{A}(G)$ and $\mathcal{F}(G)$ do not coincide. However, this is not the only such situation.  
\item   There are non-empty graphs for which the independence fractal does not coincide with the independence attractor. Indeed, the independence polynomial of the graph $G=P_2 \cup 2 K_1$ is $1+4z +5 z^2 +2z^3$ and that has a multiple root at $-1$. By Theorem~\ref{Thm_attr1}, $\mathcal{F}(G) \neq \mathcal{A}(G)$.
 \end{enumerate}
  	 \label{circle-attractor}
\end{rem} 

We now present the proof of Theorem~\ref{nocircle}
\begin{proof}[Proof of Theorem~\ref{nocircle}]
	The independence attractor of an empty graph is singleton by Remark~\ref{circle-attractor}(1), and we are done. Let $G$ be a non-empty graph. 
	 Suppose on the contrary that independence attractor $\mathcal{A}(G)$ is a circle. If $-1$ is either not a root  or is a simple root of $I_G$ then $ \mathcal{A}(G)=\mathcal{F}(G)$ by Theorem~\ref{Thm_attr1}. It follows from Lemma~\ref{circle} that $G$ is an empty graph. 
	 \par 
	 Let $-1$ be a multiple root of $I_G$. Then $-1$ is a super-attracting fixed point of $P_G$ and, it follows from Theorem~\ref{Thm_attr1} and Theorem~\ref{indfractal=Juliaset} that $\mathcal{A}(G)$ is the disjoint union of $\mathcal{J}(P_G)$ and $\cup_{m \geq 1} P_G ^{-m}(-1)$. In particular, the Julia set  $\mathcal{J}(P_G)$ is strictly contained in the circle. This means that, the Fatou set is connected and is the attracting domain corresponding to $\infty$. Thus, for every point $z$ in the Fatou set, $P^n (z) \to \infty$ as $n \to \infty$. However, $-1$, being a super-attracting fixed point of $P_G$, is in the Fatou set of $P_G$.    This is a contradiction and the proof is complete.
\end{proof}
\subsection{Four line segments}
We start with an important consequence of a situation when an independence fractal is a line segment.
\begin{theorem}
	If the independence attractor $\mathcal{A}(G)$ of a graph $G$ is a line segment then  $\mathcal{A}(G)=\mathcal{F}(G)$.
	\label{attractor=fractal}
\end{theorem}

\begin{proof} By Theorem~\ref{Thm_attr1}, $\mathcal{F}(G) \subseteq \mathcal{A}(G)$. It follows from the hypothesis of this theorem that $\mathcal{F}(G)$
 is contained in a line segment. Since $\mathcal{F}(G)$ is  the Julia set of $P_G$,  the Fatou set of $P_G$ is connected, and is the attracting domain corresponding to $\infty$ (see Theorem~\ref{poly_map}). This gives that there is no finite attracting fixed point for $P_G$. In particular, $-1$ is not a super-attracting fixed point of $P_G$. By Theorem \ref{Thm_attr1}, $\mathcal{A}(G)=\mathcal{F}(G)$.
\end{proof}

In view of Theorem~\ref{attractor=fractal}, it is enough to determine all  reduced independence polynomials  whose Julia set is a line segment. To do this, we need some useful properties of Chebyshev polynomials. The Chebyshev polynomials (of the first kind) are defined by the recurrence relation 
\begin{equation}T_{n}(z)=2zT_{n-1}(z)-T_{n-2}(z), ~\mbox{where}~  T_0(z)=1  ~\mbox{and} ~T_1(z)=z.
\label{cheby-recurrence}\end{equation} Note that  $T_2(z)=2z^2-1$ and  $T_3(z)=4z^3-3z$.  

\begin{lemma}
	Let $T_n$ denote the Chebyshev polynomial with degree $n$. Then, 
	\begin{enumerate}[label=(\roman{*})]
	
		\item $T_n(1)=1$ and  $T_n'(1)=n^2$ for all $n \geq 0$, and $T_n''(1)=\frac{2}{3}\binom{n^2}{2}$ 
	    for all $n \geq 2$.
	  \label{atone}
  %and $T_n'(-1)=(-1)^{n+1}n^2$  
% 
%		\item $T_n''(1)=\frac{2}{3}\binom{n^2}{2}$ 
%  %and $T_n ''(-1)=(-1)^n \frac{2}{3} \binom{n^2}{2}$ 
%  for all $n \geq 2$. In general, for all \(n\geq 1\) and \(1\leq p \leq n\), \[T_n^{(p)}(1)=\prod_{k=0}^{p-1}\frac{n^2-k^2}{2k+1}.\]
  	\item The coefficients of $z^n$ and $z^{n-1}$ in $T_n$ are $2^{n-1}$  and zero respectively for all \(n\geq 1\). \label{coefficients}
\end{enumerate}
	\label{derivatives}	
\end{lemma}
\begin{proof} The proofs are either straightforward or are by the method of induction.
	
	\begin{enumerate}[label=(\roman{*})]
		%\item 
		%By definition $T_n(0)=-T_{n-2}(0)$ and it follows from %$T_{1}(0)=0$
		%that $T_n(0)=0$ for all odd $n$. Since  $T_{2}(0)=-1$, %$T_n(0)=(-1)^{\frac{n}{2}}$ for all even $n$.
		
		\item	
		Observe that $T_0(1)=1$ and  $T_1(1)=1$. It follows from Equation~(\ref{cheby-recurrence}) that $T_n(1)=1$ for all $n \geq 0$.		
		%\par Note that $T_0(-1)=-1$ and $T_1(-1)=1$. Let $n$ be a natural number and $T_k (-1)=(-1)^k$ for all $k <n$. Then  $T_{n}(-1)=2(-1)(-1)^{n-1}-(-1)^{n-2} =(-1)^{n}$.   
	 Observe that $T_n'(1)=n^2$  for $n=0, 1$. For an arbitrary natural number $n$, let $T_{k}'(1)=k^2$ for all $k<n$. Since $T_{n}'(z)= 2 T_{n-1}(z)+2z T_{n-1} '(z)-T_{n-2}'(z)$,  $T_{n}'(1)=2T_{n-1}(1)+2T_{n-1}'(1)-T_{n-2}'(1)$  which is nothing but $n^2$. Similarly, it can be seen that $T_n''(1)=\frac{2}{3}\binom{n^2}{2}$ 
	 for all $n \geq 2$.

%		\item
%		As  $T_{n}'(z)= 2 T_{n-1}(z)+2z T_{n-1} '(z)-T_{n-2}'(z)$, $$T_{n}''(z)= 4 T_{n-1}'(z)+2z T_{n-1} ''(z)-T_{n-2}''(z).$$
%        Observe that $T_k ''(1)=\frac{2}{3}\binom{k^2}{2}$ for $k=2,3$. Let $n$ be a natural number and $T_k''(1) = \frac{2}{3}\binom{k^2}{2}$ for all $k<n$. Then, $T_{n}''(1)=4T_{n-1}'(1)+2T_{n-1}''(1)-T_{n-2}''(1)$. This implies that $T_{n}''(1)=4(n-1)^2+2\cdot \frac{2}{3}\binom{(n-1)^2}{2}-\frac{2}{3}\binom{(n-2)^2}{2}$ which is $\frac{2}{3} [ 6 (n-1)^2 +(n-1)^2 (n^2 -2n)-\frac{1}{2} (n-2)^2(n^2-4n +3)]$. This can be seen to be $\frac{2}{3} \binom{n^2}{2}$. 
       
\item 
It follows from Equation~(\ref{cheby-recurrence}) that the coefficient of $z^n$ in \(T_n\) is  twice the coefficient of \(z^{n-1}\)  in \(T_{n-1}\).  Since the coefficient $z$ in $ T_1 $  is $1$, it follows recursively that the coefficient of $z^n$ in $T_n$ is $2^{n-1}$. Similarly, since the coefficient  of $z^{n-1}$ in \(T_n\) is  twice the coefficient of \(z^{n-2}\)  in \(T_{n-1}\) and $T_1 (z)=z$, we have that 
the coefficient  of $z^{n-1}$ in \(T_n\) is $0$.\end{enumerate}
\end{proof}

We determine all possible  polynomials without a constant term that are conjugate to  a Chebyshev  polynomial via $z \mapsto az+1$ for every non-zero $a$.
 
\begin{lemma} 
If $P(z)=a_1z+a_2z^2+\cdots+a_nz^n$ is a polynomial  such that $aP(z)+1=T_n(az+1)$ for some non-zero $a$ then $$a_m=\frac{1}{m!}a^{m-1}\prod_{k=0}^{m-1}\frac{n^2-k^2}{2k+1}~\mbox{ for}~ 1\le m \le n.$$
 \label{poly}
\end{lemma}

\begin{proof}
Fix  $m$ such that $1\le m\le n$. Differentiating both sides of $a P(z)+1=T_n(az +1)$  $m$-times and putting $z=0$, we get $a ~m! a_m=a^m T_{n}^{(m)}(1)$  and   therefore, $a_m=a^{m-1} \frac{1}{m!}T_n^{(m)}(1).$ As $T_n^{(m)}(1)=\prod_{k=0}^{m-1}\frac{n^2-k^2}{2k+1}$ (see Appendix A, \cite{john}), %\href{https://en.wikipedia.org/wiki/Chebyshev_polynomials}{3}.
	we have $a_m=\frac{1}{m!}a^{m-1}\prod_{k=0}^{m-1}\frac{n^2-k^2}{2k+1}$ for $1\le m \le n$ as desired.
 \end{proof}
If the polynomial $P$ in Lemma~\ref{poly} is the reduced independence polynomial of some graph then   $aP(z)+1=T_n(az+1)$ is not possible for $a=\frac{5}{2}$. To prove it, we need an inequality involving the number of vertices, edges and triangles of a graph. The proof of the following lemma follows the discussion in Section 2 of \cite{NS}.
  \begin{lemma}
  	Let $G$ be a simple graph with at least two vertices. Also, let $N, E$ and $T$ denote the number of vertices, edges and triangles in $G$ respectively. If $4E > N^2$ then $T \geq \frac{E}{3N} (4E -N^2)$.
  	\label{Nograph-k5-key}
  \end{lemma}
  \begin{proof}
For an edge $e$ joining the vertices $v_1, v_2$, let $n_0 (e), n_1 (e)$ and $n_2 (e) $ be the number of vertices different from $v_1, v_2$ that are adjacent to none of, exactly one of and both of $v_1, v_2$, respectively. Then  \begin{equation}
	n_0 (e)+ n_1 (e)+n_2 (e) =N-2  ~\mbox{for every edge}~ e.
	\label{N-2}
\end{equation}
 Let $P_n$ denote the path  graph on $n$ vertices. We call it  an $n$-path. Let $T_1$ be the number of times  $P_2 \cup K_1 $, the disjoint union of $P_2$ and $K_1$,  occurs in $G$ as an induced subgraph. Then $\sum_{e \in E(G)} n_0 (e)=T_1$.
 Similarly, let $T_2$ and $T $ denote the number of $3$-paths in $G$ that is not a subgraph of any triangle, and triangles in $G$ respectively.  Taking the sum over all the edges of $G$, we get  $\sum_{e \in E(G)}n_1 (e) =2 T_2$ and  $\sum_{e \in E(G)}n_2 (e) = 3T$.  To see these equalities, note that each $3$-path in $G$ is counted exactly twice and each triangle is counted exactly thrice while the sum over all edges is taken. Thus, $ T_1 +2 T_2+3T  = \sum_{e \in E(G)}n_0 (e) + n_1 (e)+ n_2 (e)$ and using Equation~(\ref{N-2}) we get 
 \begin{equation}
T_1 + 2 T_2 +3 T  =(N-2)E.
\label{T1T2T3} 
\end{equation} 
Consider a vertex $v$ of $G$ with degree $\deg(v)$ at least $2$. For each choice of two edges incident up on $v$,   either there is a $3$-path (that is not a subgraph of any triangle) with $v$ as its  middle vertex, or there is a triangle containing $v$. There are $\binom{\deg(v)}{2}$ such distinct choices.  On the other hand, each $3$-path that is not a subgraph of any triangle is counted exactly once and each triangle is counted exactly thrice when all such choices are considered for all vertices with degree at least $2$. This gives that $	T_2 +3 T =\sum_{v \in V(G), ~\deg(v) \geq 2} \binom{\deg(v)}{2}.$    Since $ \sum_{v \in V(G), ~\deg(v) < 2} \binom{\deg(v)}{2}=0, $  $ \sum_{v \in V(G), ~\deg(v) \geq 2} \binom{\deg(v)}{2}= \sum_{v\in V(G)} \binom{\deg(v)}{2}$. Further, $-2E + \sum_{v\in V(G)}  \deg(v)^2 =-\sum_{v\in V(G)} \deg(v) + \sum_{v\in V(G)}  \deg(v)^2 = 2 \sum_{v\in V(G)}  \binom{\deg(v)}{2}$. In other words, 
 \begin{equation}
 2	T_2 +6 T  =-2 E +\sum_{v \in V(G)} \deg(v)^2.
	\label{T2T3}
\end{equation}
Eliminating $T_2$ from Equations (\ref{T1T2T3})  and (\ref{T2T3}), it is found that \begin{equation}
	3T =T_1 -NE +\sum_{v \in V(G)} \deg(v)^2.
	\label{3T3}
	\end{equation} 
Now, $\sum_{1 \leq i <j\leq N} (x_i -x_j)^2 =N\sum_{1\leq i \leq N} x_i ^2 - (\sum_{1 \leq i \leq N} x_i)^2$ for each set of natural numbers $\{x_1, x_2, \cdots, x_N\}$. Putting $\deg(v_i)=x_i$, observe that $\sum_{1 \leq i \leq N} x_i = 2E$, and consequently,  $N \sum_{v \in V(G)} \deg(v)^2 =4E^2+\sum_{1 \leq i < j\leq N} (\deg(v_i) -\deg(v_j))^2$. Putting this in Equation~(\ref{3T3}) we have $3 N T  = 4E^2 + \sum_{1 \leq i <j \leq N} (\deg(v_i) -\deg(v_j))^2 +NT_1 -N^2 E$. As $ \sum_{1 \leq i <j \leq N} (\deg(v_i) -\deg(v_j))^2 +NT_1 \geq 0$, $3 NT  \geq 4E^2 -N^2 E$, i.e., $$ T   \geq \frac{E(4E -N^2)}{3N}.$$ 
  \end{proof}
  \begin{lemma}
  	For $n \geq 2$, there is no graph with independence number $n$ such that its reduced independence polynomial $P$ satisfies $\frac{5}{2}(P(z))+1=T_n (\frac{5}{2}z+1)$.
  	\label{knot5}
  \end{lemma}
  \begin{proof}
  	Suppose on the contrary that there is a graph $G$ with independence number $n$ whose independence polynomial  $P$ satisfies $\frac{5}{2}(P(z))+1=T_n (\frac{5}{2}z+1)$. Then $P(z) = n^2z + \frac{5}{12}n^2(n^2-1)z^2 +  \frac{5}{72}n^2(n^2-1)(n^2-4)z^3  + \cdots + 5^{n-1}z^n$ by Lemma~\ref{poly}.
  	 The number of vertices $N$, edges $E$ and triangles $T$ in the complement $G^c$ of $G$ are $n^2,  \frac{5}{12}n^2 (n^2 -1), $ and $ \frac{5}{72} n^2 (n^2 -1)(n^2 -4) $ respectively. Since $4E =  \frac{20}{12} n^2 (n^2-1)$ and $ N^2= n^4$, $4E > N^2$ for all $n \geq 2$. Applying Lemma~\ref{Nograph-k5-key} to $G^c$, we get that $  \frac{5}{72} n^2 (n^2-1)(n^2 -4)  \geq  \frac{\frac{5}{12} n^2 (n^2-1) (4 \frac{5}{12} n^2 (n^2-1) -n^4)}{3n^2}$, i.e., $n^2 \leq -2$. This is absurd and the proof is complete.
  	 
  \end{proof}
   We also require two basic lemmas for the proof of Theorem~\ref{linesegment}.
  \begin{lemma}
  	\label{anidentity}
  	For all $n>2$, $6^{n-1} < \binom{n^2}{n}$.
  \end{lemma}
  
  \begin{proof}
  	This inequality is true for $n=3,4,5$. Note that
  	$$\binom{n^2}{n}=\frac{(n^2 -(n-1))(n^2-(n-2)) \cdots (n^2-1)(n^2-0) }{(n  -(n-1))(n -(n-2)) \cdots (n-1)(n -0) } =n\prod_{k=1} ^{n-1} \frac{n^2-k}{n-k}.$$ 
  	Since  $\frac{n^2 -k}{n-k} >n $ for all $n> k \geq 1$, $ \binom{n^2}{n}> n^n$.  For $n>6, n^n> 6^n > 6^{n-1}$ and this completes the proof. 
  \end{proof}

  \begin{lemma}
  	If $a$ and $b$ are integers different from $\pm 1$ such that $gcd (a, b) = 1$, then $a^n$ does not divide $b^n$ for any natural number $n$.
  	\label{division}
  \end{lemma}
  \begin{proof}
  	On the contrary, let us assume that $a^n$ divides $b^n$ for some natural number $n$. Then there exists an integer \(c \) such that \(b^n = c a^n = c a^{n-1}a\). This implies that \(a\) divides \(b^n \). In other words, \(a\) divides \( b\cdot b^{n-1}\). By Euclid's lemma, \(a\) divides \(b^{n-1}\). Proceeding in this way, we get that \(a\) divides \(b\), which is a contradiction. Thus, $a^n$ does not divide $b^n$ for any natural number $n$.
  \end{proof}
  \begin{proof}[Proof of Theorem~\ref{linesegment}]
  	Let $G$ be a graph and $P(z)=a_1 z +a_2 z^2+\cdots +a_n z^n$ be its reduced independence polynomial. Since $\mathcal{A}(G)$ is a line segment, it follows from Theorem~\ref{attractor=fractal} and Theorem~\ref{indfractal=Juliaset} that the Julia set of $P$ is a line segment. By Lemma \ref{real_points}, this line segment contains $0$ and is symmetric with respect to the real line. This leads to only two possibilities:  the Julia set $\mathcal{J}$ is a subset  of the imaginary axis or of the real axis.
  	
  	\par 
  	If  $\mathcal{J}$ is a subset of the imaginary axis then $0$ has to be its mid-point by its symmetry with respect to the origin. If it is the segment $[-i y_0,~ i y_0]$ for some $y_0 >0$, then consider $\phi(z) = \frac{i}{y_0}z$, and note that $\phi(\mathcal{J})=[-1,1]$.  Now the interval  $[-1,1]$ is completely invariant under $\phi P\phi^{-1}$, which is a polynomial with degree $n$. By Theorem 1.4.1 of \cite{beardon}, $\phi P\phi^{-1}=T_n$ or $-T_n$ where $T_n $ is the  Chebyshev polynomial with degree $n$. Now $\phi (P(z))=\pm T_n(\phi(z))$ implies that $\frac{i}{y_0}P(z)=\pm T_n(\frac{i}{y_{0}}z)$. Since the coefficient of $z^{n-1}$ in $\pm T_n$ is zero (by Lemma~\ref{derivatives} \ref{coefficients}) we have $a_{n-1} =0$. However, this cannot be true. Therefore, $\mathcal{J} \subset \mathbb{R}$ and it follows from Lemma \ref{real_points} that $\mathcal{J}= [-r,0]$ for some $r>0$.	 
  	
  	\par 
  	Now choose $\phi(z)=az+1$ where $a=\frac{2}{r}>0$ so that $\phi(\mathcal{J}) =[-1,1]$. Then the interval  $[-1,1]$ is completely invariant under $\phi P\phi^{-1}$, which is a polynomial with degree $n$. By Theorem 1.4.1  appearing in \cite{beardon}, $\phi P\phi^{-1}=-T_n$ or $T_n$ where $T_n $ is the Chebyshev polynomial with degree $n$.
  	If $\phi P =-T_n \phi$ then $ \phi(0)=-T_n(\phi(0))$ which cannot be true by Lemma~\ref{derivatives} \ref{atone}. Therefore, 
  	\begin{equation}
  		\phi P =T_n \phi.
  		\label{conju}
  	\end{equation}   
  	
  	Since $\phi(0)=1$, it follows from Lemma~\ref{derivatives}\ref{atone} that  $a_1 =T_n '(1)=n^2$ and 
  	\begin{equation}
  		a_2 =\frac{1}{2} T_n ''(1)a= \frac{1}{2} \frac{2}{3} \binom{n^2}{2}a.
  		\label{a2}
  	\end{equation}
  	Further, since $a_2 \leq \binom{n^2}{2}$, $ a \in (0,3] $. Since  $T_n = \phi P \phi ^{-1}$, we get
  	\begin{align*}
  		T_n(z)= a \left(a_1 \left(\frac{z-1}{a}\right) + \cdots + a_n \left(\frac{z-1}{a}\right)^n \right) + 1.
  	\end{align*}
  	The leading coefficient of $T_n$ is $2^{n-1}$ (Lemma~\ref{derivatives}\ref{coefficients}). 
  	This gives $a_n=2^{n-1}a^{n-1}$. Note that $a=3$  if and only if $a_2=\binom{n^2}{2}$ meaning that there is no edge in the graph. This gives $a_n =\binom{n^2}{n}$. On the other hand, $a_n=2^{n-1}3^{n-1}=6^{n-1}$. But this is ruled out by Lemma~\ref{anidentity}. 
  	Therefore, $a \in (0,3)$.
  	\par 
  	It follows from Equation~(\ref{a2}) that $a$ is a positive rational number, say $\frac{p}{q}$ with $\gcd(p,q)=1$. 
  	Then $a_n=(\frac{2p}{q})^{n-1} $. Since $a_n$ is a natural number, $q^{n-1}$ must divide $2^{n-1}p^{n-1}$. By Lemma~\ref{division}, $q^{n-1}$ divides $2^{n-1}$ and so the only possible value of $q$ is $1$ or $2$. This implies that $a \in \{\frac{1}{2}, 1, \frac{3}{2},2, \frac{5}{2} \}$. Therefore, $P(z)=\phi^{-1} T_n \phi(z)$ where $\phi(z)=az +1$ and $a \in \{\frac{1}{2}, 1, \frac{3}{2}, 2, \frac{5}{2}  \} $. It follows from  Lemma~\ref{knot5} that $a \neq \frac{5}{2}$.  Moreover, $J=[-r, 0] = [-\frac{2}{a}, 0]$. Replacing $a$ by $\frac{k}{2}$, we get that the independence fractal is $[-\frac{4}{k},0]$ for some $k \in \{1,2,3,4\}$.  
  	%      
  	%    \begin{enumerate}[label = (\roman{*})]
  		%    	\item Let the graph $G$ has independence number $3$. Then $P_G(z) = 9z + 30z^2 + 25z^3$. The complement of $G$, denoted by $G^c$ is a graph with $9$ vertices and $30$ edges, and it does not contain $K_4$ as a subgraph. However, by Theorem \ref{turan}, \(G^c\) can have at most \(27\) edges. This is a contradiction giving that no such graph $G$ exists.
  		%    	
  		%    	\item Let the graph $G$ has independence number $4$. Then $P_G(z) = 16z + 100z^2 + 200z^3 + 125z^4$. The graph $G^c$ has $16$ vertices, $100$ edges and does not contain $K_5$ as a subgraph. However, by Theorem \ref{turan}, \(G^c\) can have at most \(96\) edges. This is a contradiction giving that no such graph $G$  exists.
  		%    	
  		%    	\item Let the graph $G$ has independence number $5$. Then $P_G(z) = 25z + 250z^2 + 875z^3 + 1250z^4 + 625z^5$. The graph $G^c$ has $25$ vertices, $250$ edges and does not contain $K_6$ as a subgraph. By Theorem \ref{turan}, \(G^c\) can have at most \(250\) edges. Thus, $G^c$ can only be the graph $T(25, 5)$. So, the only possibility of $G$ is the complement of $T(25,5)$, which is the disjoint union of $5$ copies of $K_5$. The independence polynomial of this graph is $(1+5z)^5$, which is not equal to $P_G(z)$. Consequently, there does not exist any graph with independence number $5$ whose independence fractal is  $[-\frac{4}{5}, 0]$.
  		%    \end{enumerate} 
  \end{proof}

  \begin{rem} 
    For a given $n$, there are four polynomials $ \phi_{k}^{-1} T_n \phi_{k} (z), k=1,2,3,4$ where $\phi_{k}(z)=\frac{k}{2}z+1$ that are the possible reduced independence polynomials of some graphs whose independence attractor is a line segment. The Julia set of $\phi_{k}^{-1} T_n \phi_{k}$, which is the independence attractor of the graph, is $[-\frac{4}{k},0]$.
  		\begin{enumerate}[label=(\roman{*})]
  			\item  The coefficient of $z$, equivalently the number of vertices of the graph is $n^2$, which is independent of $k$.
  			\item The coefficient of $z^2$, equivalently the number of $2$-independent sets is $\frac{k}{12}n^2(n^2-1)$ where $k \in \{1,2,3,4\}$.
  			
  		\end{enumerate}
  	 
   \label{indpoly}
  \end{rem}

 \section{Examples: Graphs with independence number four}
 \label{examples}
It follows from Remark~\ref{indpoly} and  Lemma~\ref{poly} that 
 the reduced independence polynomials of graphs with independence number four whose independence attractor is a line segment are  $\leftindex^k{P}_G(z)=16z+20kz^2+8k^2z^3+k^3z^4$ for $k\in \{1,2,3,4\}$. The corresponding independence polynomial is 
 \begin{equation}
 	 \leftindex^k{I}_G(z) = 1+16z+20kz^2+8k^2z^3+k^3z^4.
 	 \label{indpoly-form}
 \end{equation}
  We first look for possible disconnected graphs with \(\leftindex^k{I}_G(z)\) as its independence polynomial.

If $G \cup H$ denotes the disjoint union of two graphs $G$ and $H$ then $I_{G\cup H}(z)=I_G(z)I_H(z)$. In general, the independence polynomial of a disconnected graph is the product of the independence polynomials of its components (Theorem 3.0.12, \cite{hickman2001}). We first find all the factorizations of the independence polynomials \(\leftindex^k{I}_G(z)\)  into polynomials with positive integer coefficients. Each factor is a possible independence polynomial for a component graph. Finally, we check the existence of graphs with the factor as the independence polynomial.
A disconnected graph \(G\) with independence number four has at most four components. The following proposition is a more detailed statement.
\begin{proposition}
	If $G$ is a disconnected graph with independence number four such that $ \mathcal{A}(G)$ is a line segment then it has at most three components. 
	\begin{enumerate}
		\item If the graph has three components then its independence polynomial is $(1+z)(1+3z)(1+12z+9z^2)$.
		\item If the graph has two components then its independence polynomial is $(1+4z+3z^2)(1+12z+9z^2)$, $(1+8z+8z^2)(1+8z+8z^2)$,  $(1+z)(1+15z+45z^2+27z^3)$ or $(1+3z)(1+13z+21z^2+9z^3)$.
	\end{enumerate}
	Further, the line segment is $[-1,0]$ if the independence  polynomial is  $(1+8z+8z^2)(1+8z+8z^2)$ and it is $[-\frac{4}{3},0]$ otherwise.
\label{disconnectedgraphs}
\end{proposition}
\begin{proof}  If \(G\) has four components \(G_i,\) with $n_i$ vertices for \(i = 1, 2, 3, 4\) then the independence polynomial of each \(G_i\) must be linear and,
\begin{align*}
    \leftindex^k{I}_G(z) &= (1+n_1z)(1+n_2z)(1+n_3z)(1+n_4z)\\
    &= 1 + (n_1+n_2+n_3+n_4)z + (n_1n_2+ n_1n_3+ n_1n_4+ n_2n_3+ n_2n_4+ n_3n_4)z^2 \\ &+ (n_1n_2n_3+ n_1n_2n_4+ n_1n_3n_4+ n_2n_3n_4)z^3 + n_1n_2n_3n_4z^4.
\end{align*}
Since, $\leftindex^k{I}_G(z) = 1+16z+20kz^2+8k^2z^3+k^3z^4$, every set of possible values of \((n_1, n_2, n_3, n_4)\) must satisfy \(n_1+n_2+n_3+n_4 = 16,~ n_1n_2+ n_1n_3+ n_1n_4+ n_2n_3+ n_2n_4+ n_3n_4 = 20k,~ n_1n_2n_3+ n_1n_2n_4+ n_1n_3n_4+ n_2n_3n_4 = 8k^2~ \text{and}~ n_1n_2n_3n_4 = k^3\) 
for some \(k \in \{1, 2, 3, 4\}\). Clearly, there is no value of \((n_1, n_2, n_3, n_4)\) satisfying $n_1 +n_2 +n_3 +n_4 =16$ and $n_1 n_2 n_3 n_4 =k^3$  simultaneously for any $k=1,2,3,4$. Therefore, \(G\)   has at most three components.
\begin{enumerate}
	\item 
If \(G\) has three components \(G_i \) with \(n_i\) vertices for \(i = 1, 2, 3\), then two of them say, \(G_1\) and \(G_2\) have independence number one and \(G_3\) has independence number two. Let \(\leftindex^k{I}_{G_1}(z) = 1+n_1z\), \(\leftindex^k{I}_{G_2}(z) = 1+n_2z\) and \(\leftindex^k{I}_{G_3}(z) = 1+n_3z+mz^2\)  for some natural number \(m\). Then, 
\begin{align*}
\leftindex^k{I}_G(z) &= (1+n_1z)(1+n_2z)(1+n_3z+mz^2) \\
&= 1 + (n_1+n_2+n_3)z + (n_1n_2+n_1n_3+n_2n_3+m)z^2\\
&+ (n_1n_2n_3+mn_1+mn_2)z^3 + mn_1n_2z^4.
\end{align*}
Every set of possible values of \((m,n_1,n_2,n_3)\) must satisfy  \(n_1+n_2+n_3 = 16,~ n_1n_2+n_1n_3+n_2n_3+m = 20k,~ n_1n_2n_3+mn_1+mn_2 = 8k^2,~ \text{and}~ mn_1n_2 = k^3\)
for some \(k \in \{1, 2, 3, 4\}\) (see Equation(\ref{indpoly-form})).
\par
For \(k=1\),   \((m, n_1, n_2) = (1,1,1)\) (as $mn_1 n_2 =1$). This implies that \(n_3=14\) and therefore $n_1 n_2 n_3 +mn_1 +mn_2 =16$. But this ought to be $8$.

 For the remaining values of \(k\), we first list all the values of \((m, n_1, n_2)\)  satisfying  $mn_1 n_2=k^3$, discarding the cases where any of \(n_1\) or \(n_2\) is greater than or equal to \(16\). Then we find $n_3$ using $n_1 +n_2 +n_3 =16$ and finally we check if the other equations are satisfied  or not. This is done for $k=2,3,4$ in Tables \ref{3components-k=2}, \ref{3components-k=3}, \ref{3components-k=4} respectively.
\par  For $k =2$, $n_1n_2n_3 +mn_1 +mn_2 = 32$. But it is clear from Table \ref{3components-k=2} that  this is not possible. 
% For k = 2
\begin{table}[h!]
\centering
    \caption{Possible values of $m, n_1, n_2$ and $n_3$ for $k=2$: $n_1n_2n_3 + mn_1 +  mn_2 =32$.}
    \begin{tabular}{|c|c|c|c|c|c|}
    \hline
      $m$ & $n_1$ & $n_2$ & $n_3 =16-n_1 -n_2$ & $n_1n_2n_3 + mn_1 +  mn_2$ &   Independence polynomial  \\
      \hline
       8 & 1 & 1 & 14 & 30 &  Not possible\\    \hline
       1 & 8 & 1 & 7  & 65 &  Not possible\\    \hline
       1 & 1 & 8 & 7  & 65 &  Not possible\\    \hline
       1 & 4 & 2 & 10  & 86 & Not possible\\    \hline
       1 & 2 & 4 & 10 & 86 & Not possible\\    \hline
       4 & 1 & 2 & 13 & 38 & Not possible\\    \hline
       2 & 1 & 4 & 11 & 54 &  Not possible\\    \hline
       4 & 2 & 1 & 13 & 38 &  Not possible\\    \hline
       2 & 4 & 1 & 11 & 54 &  Not possible\\    \hline
       2 & 2 & 2 & 12 & 56 &  Not possible\\   \hline
    \end{tabular}
    \label{3components-k=2}
\end{table}
\FloatBarrier

For $k=3, n_1 n_2 n_3+mn_1 +mn_2 =72$ and $n_1n_2 + n_1n_3 + m = 60$. Table ~\ref{3components-k=3} shows that this is possible only  when $(m,n_1,n_2,n_3) \in \{(9,1,3,12), (9,3,1,12) \}$. The resulting independence polynomial is $(1+z)(1+3z)(1+12z +9z^2).$
%For k = 3
\begin{table}[h!]
\centering
    \caption{Possible values of $m, n_1, n_2$ and $n_3$ for $k=3$: $n_1n_2n_3 +mn_1 +mn_2=72 $ and $n_1n_2 + n_1n_3 + n_2n_3 + m =60$.}
    \begin{tabular}{|c|c|c|c|c|c|c|}
    \hline
      $m$ & $n_1$ & $n_2$ & $n_3=16$ & $n_1n_2n_3 + $ & $n_1n_2 + n_1n_3 + $ & Independence \\
          &      &        & $-n_1 -n_2$   &   $ mn_1 +mn_2$  & $ n_2n_3 + m$  &  polynomial \\
      \hline
      27 & 1 & 1 & 14 & 68 & - & Not possible\\
      \hline
      1 & 3 & 9 & 4  & 120 & - & Not possible\\
         \hline
       1 & 9 & 3 & 4  & 120 & - & Not possible\\
          \hline
       3 & 1 & 9 & 6 & 84 & - & Not possible\\
          \hline
       9 & 1 & 3 & 12 & 72 & 60 & $(1+z)(1+3z)$\\
         &   &    &   &    &    & $(1+12z+9z^2)$\\ 
            \hline
       3 & 9 & 1 & 6 & 84 & - & Not possible\\
          \hline
       9 & 3 & 1 & 12 & 72 & 60 & $(1+3z)(1+z) $\\
             &   &    &   &    &    & $(1+12z+9z^2)$\\ 
                \hline
       3 & 3 & 3 & 10 & 108 & - & Not possible\\
       \hline
    \end{tabular}
    \label{3components-k=3}
\end{table}
%\FloatBarrier
  For $k = 4 $, we have $n_1 n_2 n_3 +mn_1 +mn_2 =128$ and $n_1n_2 + n_1 n_3+ n_2n_3 + m=80$. As evident from Table~\ref{3components-k=4}, this is never possible. 
  \begin{table}[h!]
\centering
    \caption{Possible values of $m, n_1, n_2$ and $n_3$ for $k=4$: $n_1 n_2 n_3 +mn_1 +mn_2 =128$ and $n_1n_2 + n_1 n_3+ n_2n_3 + m=80$.}
    \begin{tabular}{|c|c|c|c|c|c|c|}
    \hline
      $m$ & $n_1$ & $n_2$ & $n_3 =16$ & $n_1n_2n_3 + mn_1 $ & $n_1n_2 + n_1n_3 $ &  Independence \\
          &      &       &  $-n_1 -n_2$  & $+ mn_2$   & $+ n_2n_3 + m$   & polynomial \\
      \hline
       64 & 1 & 1 & 14 & 142 & - & Not possible\\ \hline
       %1 & 64 & 1 & -  & - & - & Not permissible\\
       %1 & 1 & 64 & -  & - & - & Not permissible\\
       32 & 2 & 1 & 13  & 122 & - &  Not possible\\ \hline
       %2 & 32 & 1 & -  & - & - & Not permissible\\
       %2 & 1 & 32 & - & - & - & Not permissible\\
       32 & 1 & 2 & 13 & 122 & - &  Not possible\\ \hline
       %1 & 32 & 2 & - & - & - & Not permissible\\
       %1 & 2 & 32 & - & - & - & Not permissible\\
       16 & 4 & 1 & 11 & 124 & - &  Not possible\\ \hline
       %4 & 16 & 1 & - & - & - & Not permissible\\
       %1 & 16 & 4 & -  & - & - & Not permissible\\
       %1 & 4 & 16 & -  & - & - & Not permissible\\
       %4 & 1 & 16 & -  & - & - & Not permissible\\
       16 & 1 & 4 & 11  & 124 & - &  Not possible\\ \hline
       8 & 8 & 1 & 7 & 128 & 79 &  Not possible\\ \hline
       8 & 1 & 8 & 7 & 128 & 79 &  Not possible\\ \hline
%       8 & 1 & 1 & 14 & 30 & - & Not permissible\\ \hline
       16 & 2 & 2 & 12 & 112 & - &  Not possible\\ \hline
       %2 & 16 & 2 & - & - & - & Not permissible\\
       %2 & 2 & 16 & - & - & - & Not permissible\\
       8 & 4 & 2 & 8  & 112 & - &  Not possible\\ \hline
       8 & 2 & 4 & 8  & 112 & - &  Not possible\\ \hline
       4 & 8 & 2 & 6  & 136 & - &  Not possible\\ \hline
       2 & 8 & 4 & 4  & 152 & - &  Not possible\\ \hline
       2 & 4 & 8 & 4 & 152 & - &  Not possible\\ \hline
       4 & 2 & 8 & 6 & 136 & - &  Not possible\\ \hline
       4 & 4 & 4 & 8 & 160 & - &  Not possible\\ \hline
    \end{tabular}
    \label{3components-k=4}
\end{table}
%\FloatBarrier

\item If \(G\) has two components \(G_1\) and \(G_2\) with \(n_1\) and \(n_2\) vertices respectively, then there are two possibilities:  both of them has independence number two, or one has independence number one and the other has three.

\par \underline{Case (A)}:  Let the independence number of $ G_1 $ and $G_2$ be two. Also, let \(\leftindex^k{I}_{G_1}(z)=1+n_1z+m_1z^2\) and \(\leftindex^k{I}_{G_2}(z)=1+n_2z+m_2z^2\)  for some natural numbers \(m_1\) and \(m_2\). Then
$ \leftindex^k{I}_G(z) = (1+n_1z+m_1z^2)(1+n_2z+m_2z^2) = 1 + (n_1+n_2)z + (n_1n_2+m_1+m_2)z^2 + (n_1m_2+n_2m_1)z^3 + m_1m_2z^4$. Every set of possible values of \( (m_1, m_2, n_1, n_2)\) must satisfy  \(n_1+n_2 = 16, ~ n_1n_2+m_1+m_2 = 20k, ~ n_1m_2+ n_2m_1 = 8k^2 ~\text{and} ~ m_1m_2 = k^3\)
for some \(k \in \{1, 2, 3, 4\}\) (see Equation~(\ref{indpoly-form})). For each value of \(k\), we first list all the values of \((m_1, m_2)\) such that $m_1 m_2 =k^3$. After that we find $n_1, n_2$ using $n_1 +n_2 =16$ and $n_1 m_2 +n_2 m_1 =8k^2$ and at last, check whether $n_1 n_2 +m_1 +m_2=20k$ is satisfied or not.
\par
For \(k=1\), the only value of \((m_1, m_2)\) is \((1,1)\). But for this, $n_1 m_2 + n_2 m_1 =8 k^2$ contradicts $n_1 +n_2 =16$. For $k = 2,3,4$, it follows from $n_1 +n_2 =16$ and $n_1 m_2 + n_2 m_1=8k^2$ that $n_1 =\frac{8k^2- 16 m_1}{m_2 -m_1}$ whenever $m_1 \neq m_2$. For $k=2$, $(m_1, m_2) \in \{(1,8), (8,1), (2,4), (4,2)\}$ and the only positive integral value of $n_1$ is $16$. But this is not possible as $n_2 >0$.
\par Similarly, for $k=3$,  all the possible values of $m_1, m_2, n_1, n_2$ are provided in Table~\ref{2components-equalindnumb-k=3}. The only positive integral values of $n_1 $ are $12$ and $4$ (see the last two rows). Consequently, $(n_1, n_2) \in \{(12,4), (4,12)\}$ and $n_1 n_2 +m_1 +m_2=60$ is also satisfied. Therefore, a possible independence polynomial of \(G\) is $(1+12z+9z^2)(1+4z+3z^2)$.
%\begin{table}[h!]
%\centering
%    \caption{Possible values of $m_1, m_2, n_1$ and $n_2$ for $k=2$}
%    \begin{tabular}{|c|c|c|c|c|c|}
%    \hline
%      $m_1$ & $m_2$ & $n_1=$ & $n_2$ & $n_1n_2 + m_1$ &Independence  \\
%       		 & 		 & $\frac{8k^2- 16 m_1}{m_2 -m_1}$ &  & $ + m_2=40$ & polynomial \\
%      \hline
%       8 & 1 & $\frac{96}{7}$ & - & - & Not permissible\\ [0.25cm]
%       1 & 8 & $\frac{16}{7}$ & - & - & Not permissible\\ [0.25cm]
%       2 & 4 & 0 & -  & - & Not permissible\\
%       4 & 2 & 16 & 0  & - & Not permissible\\
%       \hline
%    \end{tabular}
%    \label{Table4}
%\end{table}
%\FloatBarrier

% For k = 3
\begin{table}[h!]
\centering
    \caption{Possible values of $m_1, m_2, n_1$ and $n_2$ for $k=3$: $ n_1 n_2 +m_1 +m_2 =60$.}
    \begin{tabular}{|c|c|c|c|c|c|}
    \hline
      $m_1$ & $m_2$ & $n_1 =\frac{72- 16 m_1}{m_2 -m_1}$ & $n_2$ & $n_1n_2 + m_1 + m_2$ & Independence  polynomial \\
 
      \hline
       1 & 27 & $\frac{28}{13}$ & - & - & Not possible\\ [0.25cm]\hline
       27 & 1 & $\frac{180}{13}$ & - & - & Not possible\\ [0.25cm] \hline
       9 & 3 & 12 & 4  & 60 & $(1+12z+9z^2)(1+4z+3z^2)$\\ \hline
       3 & 9 & 4 & 12 & 60 & $(1+4z+3z^2)(1+12z+9z^2)$\\\hline
      \end{tabular}
    \label{2components-equalindnumb-k=3}
\end{table}
\FloatBarrier
 \par For $k = 4$, we have $(m_1, m_2) \in \{(1, 64), (64,1), (32,2), (2,32), (16,4), (4,16), (8,8)\}$. As evident from Table~\ref{2components-equalindnumb-k=4}, the only possible positive integral value of $n_1$ is $8$. Consequently,  $n_2 =8$ and, a possible independence polynomial of \(G\) is $(1+8z+8z^2)(1+8z+8z^2)$.
\begin{table}[h!]
\centering
    \caption{Possible values of $m_1, m_2, n_1$ and $n_2$ for $k=4$: $n_1n_2 + m_1 +m_2 =80$.}
    \begin{tabular}{|c|c|c|c|c|c|}
    \hline
      $m_1$ & $m_2$ & $n_1 =\frac{128-16 m_1}{m_2 -m_1}, $ & $n_2$ & $n_1n_2 + m_1 +m_2$ & Independence polynomial \\
          &   & $  m_1 \neq m_2 $ &   &   &  \\
      \hline
       1 & 64 & $\frac{16}{9}$ & - & - & Not possible\\ [0.25cm] \hline
       64 & 1 & $\frac{128}{9}$ & - & - & Not possible\\ [0.25cm] \hline
       32 & 2 & $\frac{64}{5}$ & - & - & Not possible\\ [0.25cm] \hline
       2 & 32 & $\frac{16}{5}$ & - & - & Not possible\\ [0.25cm] \hline
       16 & 4 & $\frac{32}{3}$ & - & - & Not possible\\ [0.25cm] \hline
       4 & 16 & $\frac{16}{3}$ & - & - & Not possible\\ [0.25cm] \hline
       8 & 8 & 8 & 8  & 80 & $(1+8z+8z^2)(1+8z+8z^2)$\\
       \hline
    \end{tabular}
    \label{2components-equalindnumb-k=4}
\end{table}
\FloatBarrier
\par \underline{Case (B)}: 
Let  the independence numbers of $G_1$ and $G_2$ be one and  three respectively. Then \(\leftindex^k{I}_{G_1}(z) = 1+n_1z\) and let \(\leftindex^k{I}_{G_2}(z) = 1+ n_2z + m_1z^2 + m_2z^3\)  for some natural numbers \(m_1\) and \(m_2\), and  
    $\leftindex^k{I}_G(z) = (1+n_1z)(1+n_2z+m_1z^2+m_2z^3) = 1 + (n_1+n_2)z + (n_1n_2+m_1)z^2 + (n_1m_1+m_2)z^3 + n_1m_2z^4$. Every set of possible values of \((m_1, m_2, n_1, n_2) \) must satisfy \(n_1+n_2 = 16,~ n_1n_2+m_1 = 20k,~ n_1m_1+ m_2 = 8k^2 ~ \text{and}~ n_1m_2 = k^3\)
for some \(k \in \{1, 2, 3, 4\}\). For each value of \(k\), we first list all the values of \((n_1, m_2)\) such that $n_1 m_2=k^3$ discarding the cases where \(n_1\geq 16\). Then we find $n_2$ and $m_1$ using $n_1 +n_2 =16$ and $n_1 n_2 +m_1=20k$. Finally, we check whether $ n_1m_1+ m_2 = 8k^2$ is satisfied or not.
\par
For \(k=1\), the only values of \((n_1, m_2)\) is \((1,1)\). This gives \(n_2=15\) and \(m_1=5\), which contradicts $n_1 m_1 +m_2 =8$. For $k=2$, it is seen from Table~\ref{3components-unequal-k=2} that $n_1m_1 + m_2=32$ is not possible.
%For $k = 2$
\begin{table}[h!]
\centering
    \caption{Possible values of $m_1, m_2, n_1$ and $n_2$ for $k=2$: $n_1m_1 + m_2=32$. }
    \begin{tabular}{|c|c|c|c|c|c|}
    \hline
      $n_1$ & $m_2$ & $n_2=16-n_1$ & $m_1 =40-n_1 n_2$ & $n_1m_1 + m_2 $ & Independence polynomial \\
      \hline
       8 & 1 & 8 & -24 & - & Not possible\\ \hline
       1 & 8 & 15 & 25 & 33 & Not possible\\ \hline
       4 & 2 & 12 & -8  & - & Not possible\\ \hline
       2 & 4 & 14 & 12  & 28 & Not possible\\ \hline
    \end{tabular}
    \label{3components-unequal-k=2}
\end{table}
%\FloatBarrier
For $k = 3$, arguing as above we get that the possible independence polynomial are $ (1+z)(1+15z+45z^2+27z^3)$ and $(1+3z)(1+13z+21z^2+9z^3)$ (see Table ~\ref{3components-unequal-k=3}).
% For k = 3
\begin{table}[h!]
\centering
    \caption{Possible values of $m_1, m_2, n_1$ and $n_2$ for $k=3$: $n_1m_1 + m_2 =72$. }
    \begin{tabular}{|c|c|c|c|c|c|}
    \hline
      $n_1$ & $m_2$ & $n_2 =16-n_1$ & $m_1=60-n_1 n_2$ & $n_1m_1 + m_2$ & Independence polynomial \\
      \hline
       1 & 27 & 15 & 45 & 72 & $(1+z)$\\
               &   &   &   &   & $ (1+15z+45z^2+27z^3)$\\ \hline
       9 & 3 & 7 & -3 & - & Not possible\\ \hline 
       3 & 9 & 13 & 21  & 72 & $(1+3z)$\\ 
         &  &  &   &  & $(1+13z+21z^2+9z^3)$\\  \hline
    \end{tabular}
    \label{3components-unequal-k=3}
\end{table}
%\FloatBarrier
For $k=4$, it is seen from Table~\ref{3components-unequal-k=4} that $n_1m_1 + m_2 =128$ can not be true.
% For k = 4
\begin{table}[h!]
\centering
    \caption{Possible values of $m_1, m_2, n_1$ and $n_2$ for $k=4$: $n_1m_1 + m_2 =128$.}
    \begin{tabular}{|c|c|c|c|c|c|}
    \hline
      $n_1$ & $m_2$ & $n_2=16-n_1$ & $m_1 =80-n_1 n_2$ & $n_1m_1 + m_2$ &Independence polynomial \\
      \hline
       1 & 64 & 15 & 65 & 129 & Not possible\\ \hline
       2 & 32 & 14 & 52 & 136 & Not possible\\ \hline
       4 & 16 & 12 & 32  & 144 & Not possible\\ \hline
       8 & 8 & 8 & 16  & 136 & Not possible\\ \hline
    \end{tabular}
    \label{3components-unequal-k=4}
\end{table}
%\FloatBarrier
\par Since $(1+8z+8z^2)(1+8z+8z^2) =  \leftindex^4{I}_G(z)$ by Equation~\ref{indpoly-form}, the corresponding line segment is $[-1,0]$ (see Table~\ref{2components-equalindnumb-k=4}). In all other cases, the corresponding value of $k$ is $3$ (see Tables~\ref{3components-k=3},~\ref{2components-equalindnumb-k=3} and ~\ref{3components-unequal-k=3}) giving that the line segment is $[-\frac{4}{3},0]$.
This completes the proof.
\end{enumerate}
\end{proof}

In order to determine the disconnected graphs whose independence polynomials are given in Proposition~\ref{disconnectedgraphs}, we need to look for  non-isomorphic connected graphs with the following independence polynomials:  $1+12z+9z^2$,  $1+4z+3z^2$, $1+8z+8z^2$, $1+15z+45z^2+27z^3$ and $1+13z+21z^2+9z^3$.
\begin{lemma}
	Let $P_1(z)= 1+12z+9z^2, P_2 (z)=1+4z+3z^2, P_3(z)=1+8z+8z^2, P_4 (z)=1+15z+45z^2+27z^3$ and $ P_5 (z)=1+13z+21z^2+9z^3 $. If $N_i$ denotes the number of non-isomorphic connected graphs whose independence polynomial is $P_i$ then $N_i >0$ for each $i$. In particular,  $N_1 \geq 10$,  $N_2 = 1$, $N_3 \geq 25$, $N_4 \geq 4$ and  $N_5 \geq 5$.
 \label{examples-disconnected}
\end{lemma}
\begin{proof}
	\begin{enumerate}
			
		\item  Let \(G_1\) be a connected graph such that \(I_{G_1}(z)=1+12z+9z^2\). Then  the complement of \(G_1\) has $12$ vertices,  \(9\) edges but no triangles. In this case,  each  vertex of $G_1^c$ has degree at most nine. Ten  non-isomorphic examples for the complement of   \(G_1\) are given in Figure~(\ref{M2}) showing that $N_1 \geq 10$.
		
		\item  Let \(G_2\) be a connected graph such that \(I_{G_2}(z)=1+4z+3z^2\). Then the complement of \(G_2\) has $4$ vertices,  \(3\) edges but no triangles, and each of its vertex has degree at most $2$. The only such graph is a path graph on \(4\) vertices. 
			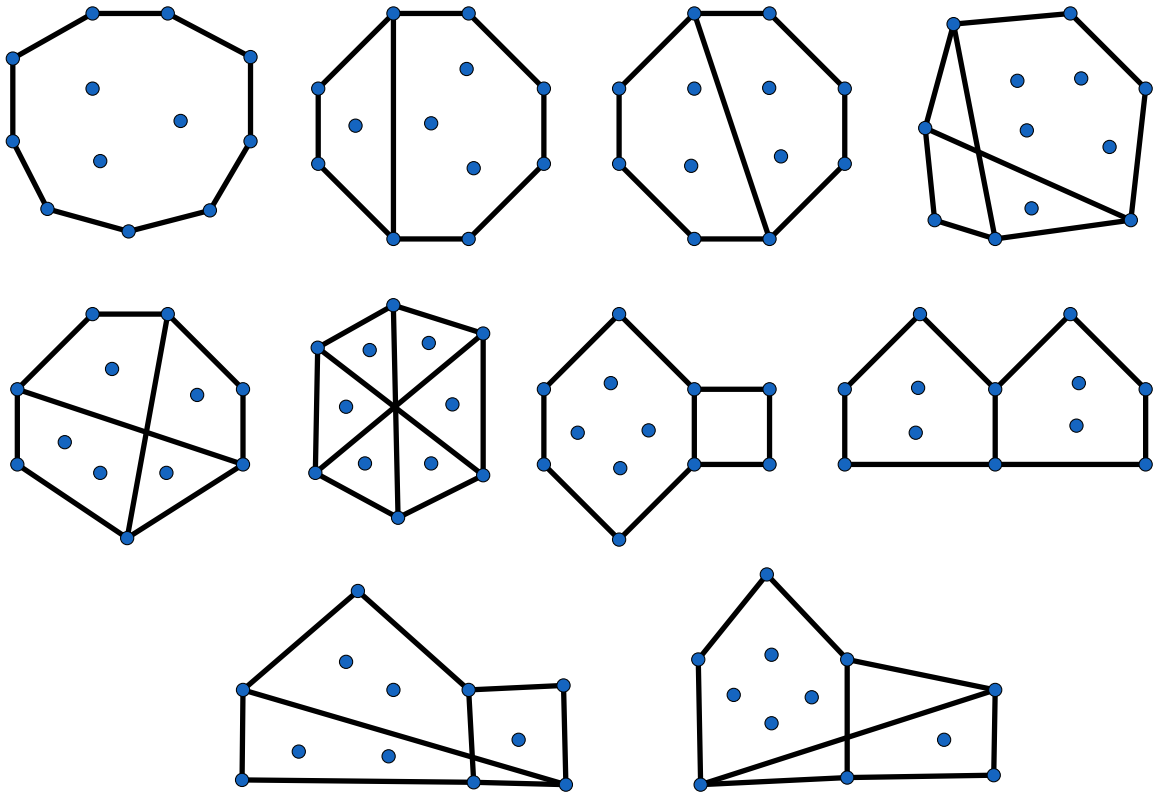
\begin{figure}[h!]
			\definecolor{rvwvcq}{rgb}{0.08235294117647059,0.396078431372549,0.7529411764705882}
			\begin{tikzpicture}
				\draw [line width=2pt] (-2,5)-- (-1,5);
				\draw [line width=2pt] (-2,5)-- (-3.06,4.4);
				\draw [line width=2pt] (-1,5)-- (0.1,4.42);
				\draw [line width=2pt] (-3.06,4.4)-- (-3.06,3.3);
				\draw [line width=2pt] (0.1,4.42)-- (0.1,3.3);
				\draw [line width=2pt] (-3.06,3.3)-- (-2.6,2.4);
				\draw [line width=2pt] (0.1,3.3)-- (-0.44,2.38);
				\draw [line width=2pt] (-2.6,2.4)-- (-1.52,2.1);
				\draw [line width=2pt] (-1.52,2.1)-- (-0.44,2.38);
				\draw [line width=2pt] (2,5)-- (3,5);
				\draw [line width=2pt] (2,5)-- (1,4);
				\draw [line width=2pt] (1,4)-- (1,3);
				\draw [line width=2pt] (1,3)-- (2,2);
				\draw [line width=2pt] (2,2)-- (3,2);
				\draw [line width=2pt] (3,2)-- (4,3);
				\draw [line width=2pt] (4,3)-- (4,4);
				\draw [line width=2pt] (4,4)-- (3,5);
				\draw [line width=2pt] (2,5)-- (2,2);
				\draw [line width=2pt] (6,5)-- (7,5);
				\draw [line width=2pt] (7,5)-- (8,4);
				\draw [line width=2pt] (8,4)-- (8,3);
				\draw [line width=2pt] (8,3)-- (7,2);
				\draw [line width=2pt] (7,2)-- (6,2);
				\draw [line width=2pt] (6,2)-- (5,3);
				\draw [line width=2pt] (5,3)-- (5,4);
				\draw [line width=2pt] (5,4)-- (6,5);
				\draw [line width=2pt] (6,5)-- (7,2);
				\draw [line width=2pt] (9.4463264390244,4.858265951219514)-- (11,5);
				\draw [line width=2pt] (11,5)-- (12,4);
				\draw [line width=2pt] (9.4463264390244,4.858265951219514)-- (9.069231414634157,3.4755841951219524);
				\draw [line width=2pt] (9.069231414634157,3.4755841951219524)-- (9.19492975609757,2.2500253658536584);
				\draw [line width=2pt] (12,4)-- (11.803170341463426,2.2500253658536584);
				\draw [line width=2pt] (9.19492975609757,2.2500253658536584)-- (10,2);
				\draw [line width=2pt] (10,2)-- (11.803170341463426,2.2500253658536584);
				\draw [line width=2pt] (9.4463264390244,4.858265951219514)-- (10,2);
				\draw [line width=2pt] (9.069231414634157,3.4755841951219524)-- (11.803170341463426,2.2500253658536584);
				\draw [line width=2pt] (-2,1)-- (-1,1);
				\draw [line width=2pt] (-2,1)-- (-3,0);
				\draw [line width=2pt] (-3,0)-- (-3,-1);
				\draw [line width=2pt] (-3,-1)-- (-1.54,-1.98);
				\draw [line width=2pt] (-1.54,-1.98)-- (0,-1);
				\draw [line width=2pt] (0,-1)-- (0,0);
				\draw [line width=2pt] (0,0)-- (-1,1);
				\draw [line width=2pt] (-3,0)-- (0,-1);
				\draw [line width=2pt] (-1,1)-- (-1.54,-1.98);
				\draw [line width=2pt] (1.9986997073170787,1.1187402926829257)-- (0.993112975609761,0.5530977560975595);
				\draw [line width=2pt] (0.993112975609761,0.5530977560975595)-- (0.9616883902439073,-1.112405268292686);
				\draw [line width=2pt] (0.9616883902439073,-1.112405268292686)-- (2.0615488780487863,-1.709472390243906);
				\draw [line width=2pt] (2.0615488780487863,-1.709472390243906)-- (3.1928339512195185,-1.1438298536585396);
				\draw [line width=2pt] (3.1928339512195185,-1.1438298536585396)-- (3.1928339512195185,0.7416452682926816);
				\draw [line width=2pt] (3.1928339512195185,0.7416452682926816)-- (1.9986997073170787,1.1187402926829257);
				\draw [line width=2pt] (1.9986997073170787,1.1187402926829257)-- (2.0615488780487863,-1.709472390243906);
				\draw [line width=2pt] (0.993112975609761,0.5530977560975595)-- (3.1928339512195185,-1.1438298536585396);
				\draw [line width=2pt] (0.9616883902439073,-1.112405268292686)-- (3.1928339512195185,0.7416452682926816);
				\draw [line width=2pt] (5,1)-- (4,0);
				\draw [line width=2pt] (4,0)-- (4,-1);
				\draw [line width=2pt] (4,-1)-- (5,-2);
				\draw [line width=2pt] (5,-2)-- (6,-1);
				\draw [line width=2pt] (6,-1)-- (6,0);
				\draw [line width=2pt] (6,0)-- (5,1);
				\draw [line width=2pt] (6,0)-- (7,0);
				\draw [line width=2pt] (7,0)-- (7,-1);
				\draw [line width=2pt] (7,-1)-- (6,-1);
				\draw [line width=2pt] (9,1)-- (8,0);
				\draw [line width=2pt] (8,0)-- (8,-1);
				\draw [line width=2pt] (9,1)-- (10,0);
				\draw [line width=2pt] (10,0)-- (10,-1);
				\draw [line width=2pt] (8,-1)-- (10,-1);
				\draw [line width=2pt] (10,-1)-- (12,-1);
				\draw [line width=2pt] (12,-1)-- (12,0);
				\draw [line width=2pt] (12,0)-- (11,1);
				\draw [line width=2pt] (11,1)-- (10,0);
				\draw [line width=2pt] (1.5273309268292736,-2.68363453658537)-- (0,-4);
				\draw [line width=2pt] (0,-4)-- (-0.012473756097556712,-5.197601365853665);
				\draw [line width=2pt] (-0.012473756097556712,-5.197601365853665)-- (3.0671356097561038,-5.229025951219518);
				\draw [line width=2pt] (3.0671356097561038,-5.229025951219518)-- (3,-4);
				\draw [line width=2pt] (3,-4)-- (1.5273309268292736,-2.68363453658537);
				\draw [line width=2pt] (3,-4)-- (4.261269853658543,-3.9406179512195174);
				\draw [line width=2pt] (4.261269853658543,-3.9406179512195174)-- (4.292694439024397,-5.260450536585372);
				\draw [line width=2pt] (4.292694439024397,-5.260450536585372)-- (3.0671356097561038,-5.229025951219518);
				\draw [line width=2pt] (0,-4)-- (4.292694439024397,-5.260450536585372);
				\draw [line width=2pt] (6.96378419512196,-2.4636624390243944)-- (6.052471219512204,-3.594947512195127);
				\draw [line width=2pt] (6.052471219512204,-3.594947512195127)-- (6.083895804878057,-5.260450536585372);
				\draw [line width=2pt] (6.083895804878057,-5.260450536585372)-- (8.032220097560986,-5.166176780487811);
				\draw [line width=2pt] (8.032220097560986,-5.166176780487811)-- (8.032220097560986,-3.594947512195127);
				\draw [line width=2pt] (8.032220097560986,-3.594947512195127)-- (6.96378419512196,-2.4636624390243944);
				\draw [line width=2pt] (8.032220097560986,-3.594947512195127)-- (10,-4);
				\draw [line width=2pt] (10,-4)-- (9.980544390243914,-5.1347521951219575);
				\draw [line width=2pt] (9.980544390243914,-5.1347521951219575)-- (8.032220097560986,-5.166176780487811);
				\draw [line width=2pt] (6.083895804878057,-5.260450536585372)-- (10,-4);
				\begin{scriptsize}
					\draw [fill=rvwvcq] (-2,5) circle (2.5pt);
					\draw [fill=rvwvcq] (-1,5) circle (2.5pt);
					\draw [fill=rvwvcq] (-3.06,4.4) circle (2.5pt);
					\draw [fill=rvwvcq] (0.1,4.42) circle (2.5pt);
					\draw [fill=rvwvcq] (-3.06,3.3) circle (2.5pt);
					\draw [fill=rvwvcq] (0.1,3.3) circle (2.5pt);
					\draw [fill=rvwvcq] (-2.6,2.4) circle (2.5pt);
					\draw [fill=rvwvcq] (-0.44,2.38) circle (2.5pt);
					\draw [fill=rvwvcq] (-1.52,2.1) circle (2.5pt);
					\draw [fill=rvwvcq] (2,5) circle (2.5pt);
					\draw [fill=rvwvcq] (3,5) circle (2.5pt);
					\draw [fill=rvwvcq] (1,4) circle (2.5pt);
					\draw [fill=rvwvcq] (1,3) circle (2.5pt);
					\draw [fill=rvwvcq] (2,2) circle (2.5pt);
					\draw [fill=rvwvcq] (3,2) circle (2.5pt);
					\draw [fill=rvwvcq] (4,3) circle (2.5pt);
					\draw [fill=rvwvcq] (4,4) circle (2.5pt);
					\draw [fill=rvwvcq] (6,5) circle (2.5pt);
					\draw [fill=rvwvcq] (7,5) circle (2.5pt);
					\draw [fill=rvwvcq] (8,4) circle (2.5pt);
					\draw [fill=rvwvcq] (8,3) circle (2.5pt);
					\draw [fill=rvwvcq] (7,2) circle (2.5pt);
					\draw [fill=rvwvcq] (6,2) circle (2.5pt);
					\draw [fill=rvwvcq] (5,3) circle (2.5pt);
					\draw [fill=rvwvcq] (5,4) circle (2.5pt);
					\draw [fill=rvwvcq] (9.4463264390244,4.858265951219514) circle (2.5pt);
					\draw [fill=rvwvcq] (11,5) circle (2.5pt);
					\draw [fill=rvwvcq] (12,4) circle (2.5pt);
					\draw [fill=rvwvcq] (9.069231414634157,3.4755841951219524) circle (2.5pt);
					\draw [fill=rvwvcq] (9.19492975609757,2.2500253658536584) circle (2.5pt);
					\draw [fill=rvwvcq] (11.803170341463426,2.2500253658536584) circle (2.5pt);
					\draw [fill=rvwvcq] (10,2) circle (2.5pt);
					\draw [fill=rvwvcq] (-2,1) circle (2.5pt);
					\draw [fill=rvwvcq] (-1,1) circle (2.5pt);
					\draw [fill=rvwvcq] (-3,0) circle (2.5pt);
					\draw [fill=rvwvcq] (-3,-1) circle (2.5pt);
					\draw [fill=rvwvcq] (-1.54,-1.98) circle (2.5pt);
					\draw [fill=rvwvcq] (0,-1) circle (2.5pt);
					\draw [fill=rvwvcq] (0,0) circle (2.5pt);
					\draw [fill=rvwvcq] (1.9986997073170787,1.1187402926829257) circle (2.5pt);
					\draw [fill=rvwvcq] (0.993112975609761,0.5530977560975595) circle (2.5pt);
					\draw [fill=rvwvcq] (0.9616883902439073,-1.112405268292686) circle (2.5pt);
					\draw [fill=rvwvcq] (2.0615488780487863,-1.709472390243906) circle (2.5pt);
					\draw [fill=rvwvcq] (3.1928339512195185,-1.1438298536585396) circle (2.5pt);
					\draw [fill=rvwvcq] (3.1928339512195185,0.7416452682926816) circle (2.5pt);
					\draw [fill=rvwvcq] (5,1) circle (2.5pt);
					\draw [fill=rvwvcq] (4,0) circle (2.5pt);
					\draw [fill=rvwvcq] (4,-1) circle (2.5pt);
					\draw [fill=rvwvcq] (5,-2) circle (2.5pt);
					\draw [fill=rvwvcq] (6,-1) circle (2.5pt);
					\draw [fill=rvwvcq] (6,0) circle (2.5pt);
					\draw [fill=rvwvcq] (7,0) circle (2.5pt);
					\draw [fill=rvwvcq] (7,-1) circle (2.5pt);
					\draw [fill=rvwvcq] (9,1) circle (2.5pt);
					\draw [fill=rvwvcq] (8,0) circle (2.5pt);
					\draw [fill=rvwvcq] (8,-1) circle (2.5pt);
					\draw [fill=rvwvcq] (10,0) circle (2.5pt);
					\draw [fill=rvwvcq] (10,-1) circle (2.5pt);
					\draw [fill=rvwvcq] (12,-1) circle (2.5pt);
					\draw [fill=rvwvcq] (12,0) circle (2.5pt);
					\draw [fill=rvwvcq] (11,1) circle (2.5pt);
					\draw [fill=rvwvcq] (1.5273309268292736,-2.68363453658537) circle (2.5pt);
					\draw [fill=rvwvcq] (0,-4) circle (2.5pt);
					\draw [fill=rvwvcq] (-0.012473756097556712,-5.197601365853665) circle (2.5pt);
					\draw [fill=rvwvcq] (3.0671356097561038,-5.229025951219518) circle (2.5pt);
					\draw [fill=rvwvcq] (3,-4) circle (2.5pt);
					\draw [fill=rvwvcq] (4.261269853658543,-3.9406179512195174) circle (2.5pt);
					\draw [fill=rvwvcq] (4.292694439024397,-5.260450536585372) circle (2.5pt);
					\draw [fill=rvwvcq] (6.96378419512196,-2.4636624390243944) circle (2.5pt);
					\draw [fill=rvwvcq] (6.052471219512204,-3.594947512195127) circle (2.5pt);
					\draw [fill=rvwvcq] (6.083895804878057,-5.260450536585372) circle (2.5pt);
					\draw [fill=rvwvcq] (8.032220097560986,-5.166176780487811) circle (2.5pt);
					\draw [fill=rvwvcq] (8.032220097560986,-3.594947512195127) circle (2.5pt);
					\draw [fill=rvwvcq] (10,-4) circle (2.5pt);
					\draw [fill=rvwvcq] (9.980544390243914,-5.1347521951219575) circle (2.5pt);
					\draw [fill=rvwvcq] (-2,4) circle (2.5pt);
					\draw [fill=rvwvcq] (-0.8295129756097523,3.569857951219513) circle (2.5pt);
					\draw [fill=rvwvcq] (-1.8979488780487774,3.03564) circle (2.5pt);
					\draw [fill=rvwvcq] (2.972861853658543,4.261198829268294) circle (2.5pt);
					\draw [fill=rvwvcq] (2.5014930731707374,3.5384333658536598) circle (2.5pt);
					\draw [fill=rvwvcq] (3.0671356097561038,2.9413662439024395) circle (2.5pt);
					\draw [fill=rvwvcq] (1.49590634146342,3.507008780487806) circle (2.5pt);
					\draw [fill=rvwvcq] (6.995208780487814,4.009802146341465) circle (2.5pt);
					\draw [fill=rvwvcq] (6,4) circle (2.5pt);
					\draw [fill=rvwvcq] (5.958197463414642,2.9727908292682934) circle (2.5pt);
					\draw [fill=rvwvcq] (7.1523317073170825,3.098489170731708) circle (2.5pt);
					\draw [fill=rvwvcq] (10.29479024390245,4.104075902439026) circle (2.5pt);
					\draw [fill=rvwvcq] (11.143254048780499,4.1355004878048796) circle (2.5pt);
					\draw [fill=rvwvcq] (11.520349073170744,3.224187512195123) circle (2.5pt);
					\draw [fill=rvwvcq] (10.420488585365865,3.4441596097560985) circle (2.5pt);
					\draw [fill=rvwvcq] (10.483337756097573,2.407148292682927) circle (2.5pt);
					\draw [fill=rvwvcq] (-1.0180604878048745,-1.112405268292686) circle (2.5pt);
					\draw [fill=rvwvcq] (-0.6095408780487767,-0.07539395121951425) circle (2.5pt);
					\draw [fill=rvwvcq] (-1.740825951219509,0.2702764878048763) circle (2.5pt);
					\draw [fill=rvwvcq] (-2.3693176585365827,-0.7038856585365879) circle (2.5pt);
					\draw [fill=rvwvcq] (-1.8979488780487774,-1.112405268292686) circle (2.5pt);
					\draw [fill=rvwvcq] (1.6844538536585418,0.5216731707317058) circle (2.5pt);
					\draw [fill=rvwvcq] (2.470068487804884,0.6159469268292669) circle (2.5pt);
					\draw [fill=rvwvcq] (2.784314341463421,-0.201092292682929) circle (2.5pt);
					\draw [fill=rvwvcq] (2.5014930731707374,-0.9867069268292711) circle (2.5pt);
					\draw [fill=rvwvcq] (1.6216046829268347,-0.9867069268292711) circle (2.5pt);
					\draw [fill=rvwvcq] (1.3702080000000052,-0.23251687804878268) circle (2.5pt);
					\draw [fill=rvwvcq] (4.889761560975617,0.08172897560975417) circle (2.5pt);
					\draw [fill=rvwvcq] (4.449817365853666,-0.5781873170731732) circle (2.5pt);
					\draw [fill=rvwvcq] (5.392554926829276,-0.5467627317073195) circle (2.5pt);
					\draw [fill=rvwvcq] (5.015459902439032,-1.0495560975609786) circle (2.5pt);
					\draw [fill=rvwvcq] (8.974957658536596,0.018879804878046805) circle (2.5pt);
					\draw [fill=rvwvcq] (11.111829463414646,0.08172897560975417) circle (2.5pt);
					\draw [fill=rvwvcq] (8.943533073170741,-0.5781873170731732) circle (2.5pt);
					\draw [fill=rvwvcq] (1.3702080000000052,-3.6263720975609806) circle (2.5pt);
					\draw [fill=rvwvcq] (3.6642027317073236,-4.663383414634152) circle (2.5pt);
					\draw [fill=rvwvcq] (0.7417162926829316,-4.820506341463421) circle (2.5pt);
					\draw [fill=rvwvcq] (1.9358505365853713,-4.883355512195128) circle (2.5pt);
					\draw [fill=rvwvcq] (2,-4) circle (2.5pt);
					\draw [fill=rvwvcq] (9.320628097560986,-4.663383414634152) circle (2.5pt);
					\draw [fill=rvwvcq] (7.026633365853668,-3.5320983414634197) circle (2.5pt);
					\draw [fill=rvwvcq] (6.52384,-4.066316292682933) circle (2.5pt);
					\draw [fill=rvwvcq] (7.56085131707318,-4.097740878048786) circle (2.5pt);
					\draw [fill=rvwvcq] (7.026633365853668,-4.443411317073177) circle (2.5pt);
					\draw [fill=rvwvcq] (11.080404878048792,-0.48391356097561217) circle (2.5pt);
				\end{scriptsize}
			\end{tikzpicture}
			\caption{Ten possibilities for the complement of $G_1$ where \(I_{G_1}(z)=1+12z +9z^2\).}
			\label{M2}
		\end{figure}
%		\FloatBarrier
		\item  Let \(G_3\) be a connected graph such that \(I_{G_3}(z)=1+8z+8z^2\). Then the complement of \(G_3\) has $8$ vertices,  \(8\) edges but no triangles, and each of its vertex has degree at most $6$. Figure ~(\ref{M_1}) provides \(25\)  non-isomorphic examples for the complement of \(G_3\). The graphs are given in the increasing order of the number of pendants in them.

		\begin{figure}[h!]
		\begin{center}
			\begin{multicols}{2}
			
			\definecolor{ududff}{rgb}{0.30196078431372547,0.30196078431372547,1}
			\begin{tikzpicture}
				\draw [line width=2pt] (-4.6,2.64)-- (-0.28,2.32);
				\draw [line width=2pt] (-4.6,2.64)-- (-0.02,1.34);
				\draw [line width=2pt] (-4.6,2.64)-- (-0.52,0.18);
				\draw [line width=2pt] (-4.6,2.64)-- (-1.14,-0.82);
				\draw [line width=2pt] (-0.28,2.32)-- (-5.06,1.72);
				\draw [line width=2pt] (-5.06,1.72)-- (-0.02,1.34);
				\draw [line width=2pt] (-5.06,1.72)-- (-0.52,0.18);
				\draw [line width=2pt] (-5.06,1.72)-- (-1.14,-0.82);
				\draw [line width=2pt] (-5,0.66)-- (-0.28,2.32);
				\draw [line width=2pt] (-5,0.66)-- (-0.02,1.34);
				\draw [line width=2pt] (-5,0.66)-- (-0.52,0.18);
				\draw [line width=2pt] (-5,0.66)-- (-1.14,-0.82);
				\draw [line width=2pt] (-0.28,2.32)-- (-4.96,-0.28);
				\draw [line width=2pt] (-0.28,2.32)-- (-3.48,-0.96);
				\draw [line width=2pt] (-0.28,2.32)-- (-2.82,-1.44);
				\draw [line width=2pt] (-4.96,-0.28)-- (-0.02,1.34);
				\draw [line width=2pt] (-4.96,-0.28)-- (-0.52,0.18);
				\draw [line width=2pt] (-3.48,-0.96)-- (-0.52,0.18);
				\draw [line width=2pt] (-2.82,-1.44)-- (-1.14,-0.82);
				\draw [line width=2pt] (-3.48,-0.96)-- (-1.14,-0.82);
				\draw [line width=2pt] (-2.82,-1.44)-- (-0.52,0.18);
				\draw [line width=2pt] (-0.02,1.34)-- (-2.82,-1.44);
				\draw [line width=2pt] (-0.02,1.34)-- (-3.48,-0.96);
				\draw [line width=2pt] (-4.96,-0.28)-- (-1.14,-0.82);
				\draw [line width=2pt] (-2.48,2.9)-- (-4.6,2.64);
				\draw [line width=2pt] (-2.48,2.9)-- (-5.06,1.72);
				\draw [line width=2pt] (-2.48,2.9)-- (-5,0.66);
				\draw [line width=2pt] (-2.48,2.9)-- (-4.96,-0.28);
				\draw [line width=2pt] (-2.48,2.9)-- (-3.48,-0.96);
				\draw [line width=2pt] (-2.48,2.9)-- (-2.82,-1.44);
				\draw [line width=2pt] (-2.48,2.9)-- (-0.28,2.32);
				\draw [line width=2pt] (-2.48,2.9)-- (-0.02,1.34);
				\draw [line width=2pt] (-2.48,2.9)-- (-0.52,0.18);
				\draw [line width=2pt] (-2.48,2.9)-- (-1.14,-0.82);
				\draw [line width=2pt] (-4.96,-0.28)-- (-4.66,-0.88);
				\draw [line width=2pt] (-3.48,-0.96)-- (-4.66,-0.88);
				\draw [line width=2pt] (-4.66,-0.88)-- (-2.82,-1.44);
				\draw [line width=2pt] (-2.82,-1.44)-- (-4.26,-1.58);
				\draw [line width=2pt] (-4.66,-0.88)-- (-4.26,-1.58);
				\draw [line width=2pt] (-3.48,-0.96)-- (-4.26,-1.58);
				\draw [line width=2pt] (-1.14,-0.82)-- (-0.76,-1.38);
				\draw [line width=2pt] (-0.76,-1.38)-- (0,-1);
				\draw [line width=2pt] (-1.14,-0.82)-- (0,-1);
				\draw [line width=2pt] (-0.52,0.18)-- (0,-1);
				\draw [line width=2pt] (-0.02,1.34)-- (0,-1);
				\begin{scriptsize}
					\draw [fill=ududff] (-0.28,2.32) circle (2.5pt);
					\draw[color=ududff] (-0.06,2.61) node {$w_2$};
					\draw [fill=ududff] (-1.14,-0.82) circle (2.5pt);
					\draw[color=ududff] (-0.9,-0.59) node {$w_5$};
					\draw [fill=ududff] (-0.52,0.18) circle (2.5pt);
					\draw[color=ududff] (-0.32,0.45) node {$w_4$};
					\draw [fill=ududff] (-0.02,1.34) circle (2.5pt);
					\draw[color=ududff] (0.22,1.59) node {$w_3$};
					\draw [fill=ududff] (-5,0.66) circle (2.5pt);
					\draw[color=ududff] (-5.52,0.77) node {$w_{13}$};
					\draw [fill=ududff] (-4.96,-0.28) circle (2.5pt);
					\draw[color=ududff] (-5.48,-0.15) node {$w_{12}$};
					\draw [fill=ududff] (-3.48,-0.96) circle (2.5pt);
					\draw[color=ududff] (-3.78,-0.67) node {$w_{10}$};
					\draw [fill=ududff] (-2.82,-1.44) circle (2.5pt);
					\draw[color=ududff] (-2.56,-1.55) node {$w_8$};
					\draw [fill=ududff] (-5.06,1.72) circle (2.5pt);
					\draw[color=ududff] (-5.56,1.85) node {$w_{14}$};
					\draw [fill=ududff] (-4.6,2.64) circle (2.5pt);
					\draw[color=ududff] (-5.12,2.77) node {$w_{15}$};
					\draw [fill=ududff] (-2.48,2.9) circle (2.5pt);
					\draw[color=ududff] (-2.42,3.25) node {$w_1$};
					\draw [fill=ududff] (-4.66,-0.88) circle (2.5pt);
					\draw[color=ududff] (-5.08,-0.83) node {$w_{11}$};
					\draw [fill=ududff] (-4.26,-1.58) circle (2.5pt);
					\draw[color=ududff] (-4.64,-1.43) node {$w_9$};
					\draw [fill=ududff] (-0.76,-1.38) circle (2.5pt);
					\draw[color=ududff] (-1.2,-1.4) node {$w_7$};
					\draw [fill=ududff] (0,-1) circle (2.5pt);
					\draw[color=ududff] (0.29,-1.07) node {$w_6$};
				\end{scriptsize}
			\end{tikzpicture}
				\caption{\label{Com G4} Complement of $G_4$}
					\definecolor{xdxdff}{rgb}{0.49019607843137253,0.49019607843137253,1}
				\definecolor{ududff}{rgb}{0.30196078431372547,0.30196078431372547,1}
				\begin{tikzpicture}
					\draw [line width=2pt] (-6,4)-- (-7,3);
					\draw [line width=2pt] (-6,4)-- (-5,3);
					\draw [line width=2pt] (-5,3)-- (-5.06,1.14);
					\draw [line width=2pt] (-7,3)-- (-7.02,1.14);
					\draw [line width=2pt] (-7.02,1.14)-- (-6,0);
					\draw [line width=2pt] (-5.06,1.14)-- (-6,0);
					\draw [line width=2pt] (-6,4)-- (-7.02,1.14);
					\draw [line width=2pt] (-6,4)-- (-5.06,1.14);
					\draw [line width=2pt] (-7,3)-- (-6,0);
					\draw [line width=2pt] (-5,3)-- (-6,0);
					\draw [line width=2pt] (-7,4)-- (-7,3);
					\draw [line width=2pt] (-5,4)-- (-5,3);
					\draw [line width=2pt] (-4,4)-- (-4,3);
					\draw [line width=2pt] (-5,3)-- (-4,3);
					\draw [line width=2pt] (-4,4)-- (-5,3);
					\draw [line width=2pt] (-4,3)-- (-4.04,1.12);
					\draw [line width=2pt] (-5.06,1.14)-- (-4.04,1.12);
					\draw [line width=2pt] (-5.06,1.14)-- (-5.04,0);
					\draw [line width=2pt] (-7.02,1.14)-- (-7,0);
					\draw [line width=2pt] (-7,3)-- (-5,3);
					\draw [line width=2pt] (-7.02,1.14)-- (-5.06,1.14);
					\begin{scriptsize}
						\draw [fill=ududff] (-6,4) circle (2.5pt);
						\draw[color=ududff] (-5.96,4.53) node {$v_2$};
						\draw [fill=ududff] (-7,3) circle (2.5pt);
						\draw[color=ududff] (-7.4,2.91) node {$v_5$};
						\draw [fill=ududff] (-5,3) circle (2.5pt);
						\draw[color=ududff] (-4.72,2.85) node {$v_6$};
						\draw [fill=ududff] (-7.02,1.14) circle (2.5pt);
						\draw[color=ududff] (-7.36,1.09) node {$v_8$};
						\draw [fill=ududff] (-6,0) circle (2.5pt);
						\draw[color=ududff] (-5.78,-0.25) node {$v_{12}$};
						\draw [fill=ududff] (-5.06,1.14) circle (2.5pt);
						\draw[color=ududff] (-4.78,1.01) node {$v_9$};
						\draw [fill=ududff] (-7,4) circle (2.5pt);
						\draw[color=ududff] (-7,4.51) node {$v_1$};
						\draw [fill=ududff] (-5,4) circle (2.5pt);
						\draw[color=ududff] (-4.92,4.43) node {$v_3$};
						\draw [fill=xdxdff] (-7,0) circle (2.5pt);
						\draw[color=xdxdff] (-6.94,-0.21) node {$v_{11}$};
						\draw [fill=xdxdff] (-5.04,0) circle (2.5pt);
						\draw[color=xdxdff] (-4.84,-0.23) node {$v_{13}$};
						\draw [fill=ududff] (-4.04,1.12) circle (2.5pt);
						\draw[color=ududff] (-3.7,0.91) node {$v_{10}$};
						\draw [fill=ududff] (-4,3) circle (2.5pt);
						\draw[color=ududff] (-3.64,2.93) node {$v_7$};
						\draw [fill=ududff] (-4,4) circle (2.5pt);
						\draw[color=ududff] (-3.94,4.51) node {$v_4$};
					\end{scriptsize}
				\end{tikzpicture}
				\caption{\label{Com G3} Complement of $G_5$}
			\end{multicols}		
		\end{center}
	\end{figure}
	
	%		\FloatBarrier
	\item 	 Let \(G_4\) be a connected graph such that \(I_{G_4}(z)=1+15z+45z^2+27z^3 \). Then the complement of \(G_4\) has $15$ vertices,  \(45\) edges and $27$ triangles. Further, each of its vertex has degree at most $13$. Figure~(\ref{Com G4}) provides such a graph.  Denote the edge joining two vertices $w_i, w_j$ by $w_i w_j$. Three more non-isomorphic graphs with the same independence polynomial as that of $G_4$ can be obtained by replacing the edges  $w_7 w_5$ and $w_7 w_6$ by (i) $w_7 w_5, w_7 w_{8}$, (ii) $w_7 w_8, w_7 w_9$,  or (iii) $w_7 w_{12}, w_7  w_{11}$ respectively. Thus, $N_4 \geq 4$.  
		\item Let \(G_5\) be a connected graph such that \(I_{G_5}(z)=1+13z+21z^2+9z^3\). Then the complement of \(G_5\) has $13$ vertices,  \(21\) edges and $9$ triangles. Further, each of its vertex has degree at most $11$. Figure~(\ref{Com G3}) provides such a graph. Four more non-isomorphic graphs with the same independence polynomial as that of $G_5$ can be obtained by replacing the edges  $v_4 v_6$ and $v_4 v_7$ by (i) $v_4 v_7, v_4 v_{10}$, (ii) $v_4 v_1, v_4 v_5$, (iii) $v_4 v_5, v_4  v_8$, or (iv) $v_4 v_3, v_4 v_8$ respectively. Thus, $N_5 \geq 5$.
			\end{enumerate}

		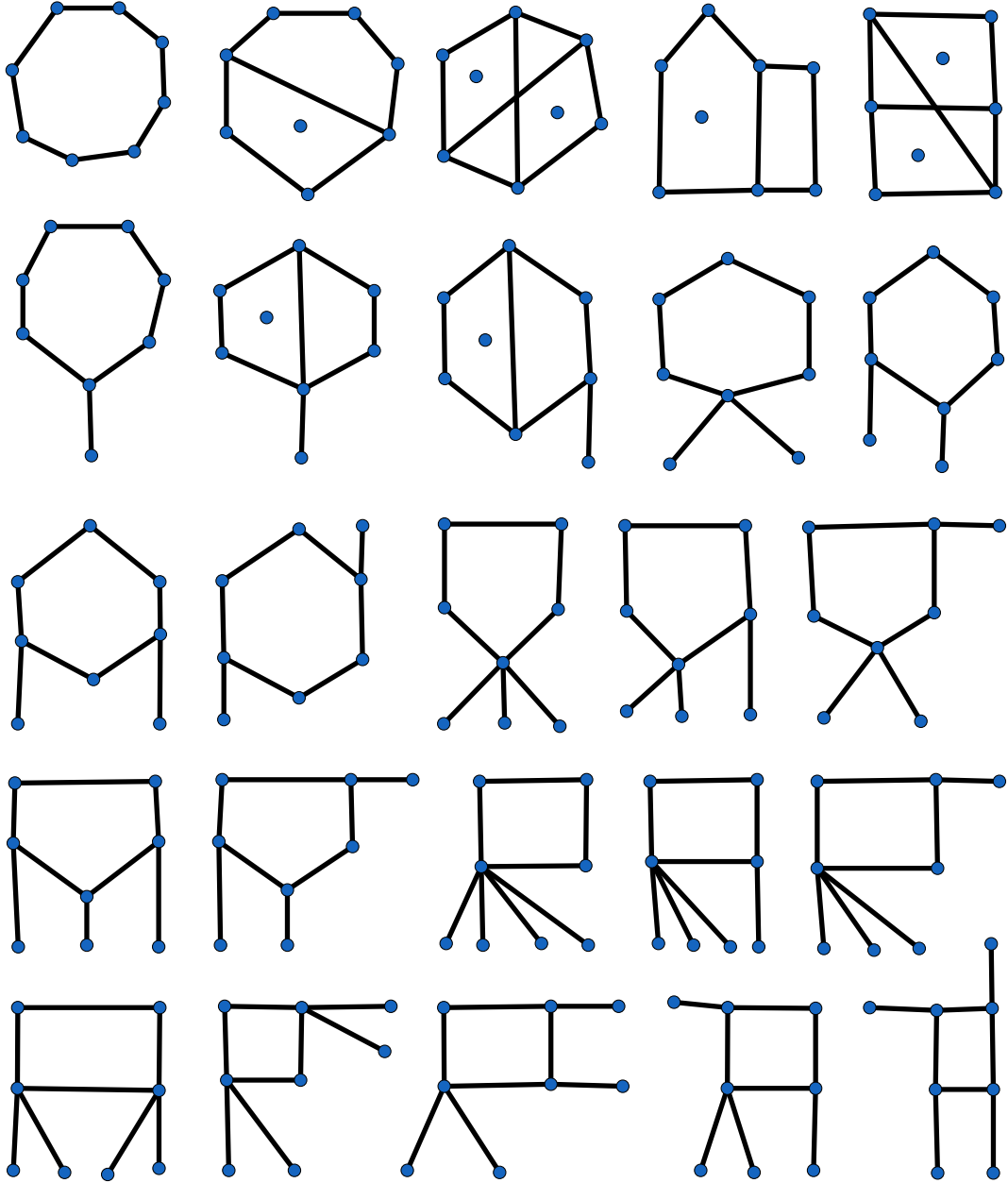
\begin{figure}[h!]
		\definecolor{rvwvcq}{rgb}{0.08235294117647059,0.396078431372549,0.7529411764705882}
		\begin{tikzpicture}
			\draw [line width=2pt] (-2.9809007718018217,4.790154652324867)-- (-4,4);
			\draw [line width=2pt] (-4,4)-- (-3.9461834213678677,3.165654583542984);
			\draw [line width=2pt] (-3.9461834213678677,3.165654583542984)-- (-2.933813813286405,2.6241545606156893);
			\draw [line width=2pt] (-2.933813813286405,2.6241545606156893)-- (-1.9920746429780674,3.2598285005738177);
			\draw [line width=2pt] (-1.9920746429780674,3.2598285005738177)-- (-2,4);
			\draw [line width=2pt] (-2,4)-- (-2.9809007718018217,4.790154652324867);
			\draw [line width=2pt] (-3.9461834213678677,3.165654583542984)-- (-4,2);
			\draw [line width=2pt] (-1.9920746429780674,3.2598285005738177)-- (-2,2);
			\draw [line width=2pt] (-0.03796586458826717,4.743067693809451)-- (-1.1209659104428553,4.013219836820489);
			\draw [line width=2pt] (-1.1209659104428553,4.013219836820489)-- (-1.0974224311851468,2.9302197909658996);
			\draw [line width=2pt] (-1.0974224311851468,2.9302197909658996)-- (-0.03796586458826717,2.3651762887808965);
			\draw [line width=2pt] (-0.03796586458826717,2.3651762887808965)-- (0.8566863472046534,2.906676311708191);
			\draw [line width=2pt] (0.8566863472046534,2.906676311708191)-- (0.833142867946945,4.036763316078197);
			\draw [line width=2pt] (0.833142867946945,4.036763316078197)-- (-0.03796586458826717,4.743067693809451);
			\draw [line width=2pt] (-1.0974224311851468,2.9302197909658996)-- (-1.0974224311851468,2.0591110584306866);
			\draw [line width=2pt] (0.833142867946945,4.036763316078197)-- (0.8566863472046534,4.790154652324867);
			\draw [line width=2pt] (2.010316830832367,4.813698131582576)-- (3.6583603788719574,4.813698131582576);
			\draw [line width=2pt] (2.010316830832367,4.813698131582576)-- (2.010316830832367,3.6365241686971532);
			\draw [line width=2pt] (3.6583603788719574,4.813698131582576)-- (3.6112734203565404,3.612980689439445);
			\draw [line width=2pt] (2.010316830832367,3.6365241686971532)-- (2.834338604852162,2.8595893531927743);
			\draw [line width=2pt] (2.834338604852162,2.8595893531927743)-- (3.6112734203565404,3.612980689439445);
			\draw [line width=2pt] (2.834338604852162,2.8595893531927743)-- (2,2);
			\draw [line width=2pt] (2.834338604852162,2.8595893531927743)-- (2.8578820841098707,2.0120240999152696);
			\draw [line width=2pt] (2.834338604852162,2.8595893531927743)-- (3.634816899614249,1.9649371413998526);
			\draw [line width=2pt] (4.553012590664878,4.790154652324867)-- (6.248143097219885,4.790154652324867);
			\draw [line width=2pt] (4.553012590664878,4.790154652324867)-- (4.576556069922587,3.5894372101817362);
			\draw [line width=2pt] (6.248143097219885,4.790154652324867)-- (6.318773534993011,3.5423502516663192);
			\draw [line width=2pt] (4.576556069922587,3.5894372101817362)-- (5.306403926911548,2.8360458739350656);
			\draw [line width=2pt] (5.306403926911548,2.8360458739350656)-- (6.318773534993011,3.5423502516663192);
			\draw [line width=2pt] (5.306403926911548,2.8360458739350656)-- (4.576556069922587,2.176828454719229);
			\draw [line width=2pt] (5.306403926911548,2.8360458739350656)-- (5.3534908854269645,2.1061980169461036);
			\draw [line width=2pt] (6.318773534993011,3.5423502516663192)-- (6.318773534993011,2.129741496203812);
			\draw [line width=2pt] (7.142795309012806,4.766611173067159)-- (8.908556253340938,4.813698131582576);
			\draw [line width=2pt] (7.142795309012806,4.766611173067159)-- (7.213425746785932,3.518806772408611);
			\draw [line width=2pt] (8.908556253340938,4.813698131582576)-- (8.908556253340938,3.565893730924028);
			\draw [line width=2pt] (7.213425746785932,3.518806772408611)-- (8.108077958578852,3.07148066651215);
			\draw [line width=2pt] (8.108077958578852,3.07148066651215)-- (8.908556253340938,3.565893730924028);
			\draw [line width=2pt] (8.108077958578852,3.07148066651215)-- (7.354686622332182,2.082654537688395);
			\draw [line width=2pt] (8.108077958578852,3.07148066651215)-- (8.720208419279272,2.035567579172978);
			\draw [line width=2pt] (8.908556253340938,4.813698131582576)-- (9.826751944391567,4.790154652324867);
			\draw [line width=2pt] (-4.040357338398701,1.1644588466377652)-- (-2.0627050807511926,1.1880023258954735);
			\draw [line width=2pt] (-4.040357338398701,1.1644588466377652)-- (-4.0639008176564095,0.31689359336026063);
			\draw [line width=2pt] (-2.0627050807511926,1.1880023258954735)-- (-2.015618122235776,0.3404370726179691);
			\draw [line width=2pt] (-4.0639008176564095,0.31689359336026063)-- (-3.0279877303172387,-0.43649774288641);
			\draw [line width=2pt] (-3.0279877303172387,-0.43649774288641)-- (-2.015618122235776,0.3404370726179691);
			\draw [line width=2pt] (-4.0639008176564095,0.31689359336026063)-- (-3.9932703798832847,-1.1428021206176637);
			\draw [line width=2pt] (-3.0279877303172387,-0.43649774288641)-- (-3.0279877303172387,-1.1192586413599552);
			\draw [line width=2pt] (-2.015618122235776,0.3404370726179691)-- (-2.015618122235776,-1.1428021206176637);
			\draw [line width=2pt] (-1.1209659104428553,1.211545805153182)-- (0.6918819924006944,1.211545805153182);
			\draw [line width=2pt] (-1.1209659104428553,1.211545805153182)-- (-1.168052868958272,0.3404370726179691);
			\draw [line width=2pt] (-1.168052868958272,0.3404370726179691)-- (-0.20277021939222623,-0.3423238258555762);
			\draw [line width=2pt] (-0.20277021939222623,-0.3423238258555762)-- (0.7154254716584028,0.26980663484484374);
			\draw [line width=2pt] (0.6918819924006944,1.211545805153182)-- (0.7154254716584028,0.26980663484484374);
			\draw [line width=2pt] (-1.168052868958272,0.3404370726179691)-- (-1.1445093897005638,-1.1192586413599552);
			\draw [line width=2pt] (-0.20277021939222623,-0.3423238258555762)-- (-0.20277021939222623,-1.1192586413599552);
			\draw [line width=2pt] (0.6918819924006944,1.211545805153182)-- (1.5629907249359065,1.211545805153182);
			\draw [line width=2pt] (2.504729895244244,1.1880023258954735)-- (4.011512567737584,1.211545805153182);
			\draw [line width=2pt] (2.504729895244244,1.1880023258954735)-- (2.5282733745019526,-0.012715116247657764);
			\draw [line width=2pt] (2.5282733745019526,-0.012715116247657764)-- (4,0);
			\draw [line width=2pt] (4,0)-- (4.011512567737584,1.211545805153182);
			\draw [line width=2pt] (2.5282733745019526,-0.012715116247657764)-- (2.033860310090075,-1.0957151621022467);
			\draw [line width=2pt] (2.5282733745019526,-0.012715116247657764)-- (2.551816853759661,-1.1192586413599552);
			\draw [line width=2pt] (2.5282733745019526,-0.012715116247657764)-- (3.3758386277794563,-1.0957151621022467);
			\draw [line width=2pt] (2.5282733745019526,-0.012715116247657764)-- (4.035056046995292,-1.1192586413599552);
			\draw [line width=2pt] (4.906164779530505,1.1880023258954735)-- (6.412947452023845,1.211545805153182);
			\draw [line width=2pt] (4.906164779530505,1.1880023258954735)-- (4.929708258788213,0.05791532152546761);
			\draw [line width=2pt] (4.929708258788213,0.05791532152546761)-- (6.412947452023845,0.05791532152546761);
			\draw [line width=2pt] (6.412947452023845,1.211545805153182)-- (6.412947452023845,0.05791532152546761);
			\draw [line width=2pt] (4.929708258788213,0.05791532152546761)-- (5.023882175819047,-1.0957151621022467);
			\draw [line width=2pt] (4.929708258788213,0.05791532152546761)-- (6.03625178390051,-1.1428021206176637);
			\draw [line width=2pt] (4.929708258788213,0.05791532152546761)-- (5.518295240230924,-1.1192586413599552);
			\draw [line width=2pt] (6.412947452023845,0.05791532152546761)-- (6.436490931281553,-1.1428021206176637);
			\draw [line width=2pt] (7.260512705301348,1.1880023258954735)-- (8.932099732598648,1.211545805153182);
			\draw [line width=2pt] (7.260512705301348,1.1880023258954735)-- (7.260512705301348,-0.03625859550536622);
			\draw [line width=2pt] (7.260512705301348,-0.03625859550536622)-- (8.955643211856355,-0.03625859550536622);
			\draw [line width=2pt] (8.955643211856355,-0.03625859550536622)-- (8.932099732598648,1.211545805153182);
			\draw [line width=2pt] (7.260512705301348,-0.03625859550536622)-- (7.354686622332182,-1.1663455998753722);
			\draw [line width=2pt] (7.260512705301348,-0.03625859550536622)-- (8.060991000063435,-1.1898890791330807);
			\draw [line width=2pt] (7.260512705301348,-0.03625859550536622)-- (8.696664940021563,-1.1663455998753722);
			\draw [line width=2pt] (8.932099732598648,1.211545805153182)-- (9.826751944391567,1.1880023258954735);
			\draw [line width=2pt] (-1.0616144770680143,11.421341427858838)-- (-0.39983089953494266,12.010932677380072);
			\draw [line width=2pt] (3.0120809312349617,12.02485185871854)-- (1.9861131987734724,11.42134142785887);
			\draw [line width=2pt] (1.9861131987734724,11.42134142785887)-- (2,10);
			\draw [line width=2pt] (2,10)-- (3.042256452777948,9.5504590921938);
			\draw [line width=2pt] (4.219101792954362,10.455724738483347)-- (3.042256452777948,9.5504590921938);
			\draw [line width=2pt] (4.007873142153467,11.632570078659734)-- (4.219101792954362,10.455724738483347);
			\draw [line width=2pt] (3.0120809312349617,12.02485185871854)-- (4.007873142153467,11.632570078659734);
			\draw [line width=2pt] (4.007873142153467,11.632570078659734)-- (2,10);
			\draw [line width=2pt] (3.0120809312349617,12.02485185871854)-- (3.042256452777948,9.5504590921938);
			\draw [line width=2pt] (5.727877870103613,12.055027380261524)-- (5.064016396157944,11.270463820143945);
			\draw [line width=2pt] (5.727877870103613,12.055027380261524)-- (6.452090387135253,11.270463820143913);
			\draw [line width=2pt] (5.064016396157944,11.270463820143945)-- (5.033840874614957,9.490108049107826);
			\draw [line width=2pt] (5.033840874614957,9.490108049107826)-- (6.42191486559227,9.520283570650813);
			\draw [line width=2pt] (6.452090387135253,11.270463820143913)-- (6.42191486559227,9.520283570650813);
			\draw [line width=2pt] (7.206478425709878,11.240288298600928)-- (6.452090387135253,11.270463820143913);
			\draw [line width=2pt] (6.42191486559227,9.520283570650813)-- (7.236653947252865,9.520283570650813);
			\draw [line width=2pt] (7.206478425709878,11.240288298600928)-- (7.236653947252865,9.520283570650813);
			\draw [line width=2pt] (8,12)-- (8.021217507370473,10.697128910827228);
			\draw [line width=2pt] (8.021217507370473,10.697128910827228)-- (8.081568550456444,9.459932527564842);
			\draw [line width=2pt] (8,12)-- (9.711046713777634,11.964500815632569);
			\draw [line width=2pt] (9.711046713777634,11.964500815632569)-- (9.771397756863603,10.666953389284242);
			\draw [line width=2pt] (9.771397756863603,10.666953389284242)-- (9.771397756863607,9.490108049107828);
			\draw [line width=2pt] (8.081568550456444,9.459932527564842)-- (9.771397756863607,9.490108049107828);
			\draw [line width=2pt] (8.021217507370473,10.697128910827228)-- (9.771397756863603,10.666953389284242);
			\draw [line width=2pt] (8,12)-- (9.771397756863607,9.490108049107828);
			\draw [line width=2pt] (-2.0110116053141995,-4.264575426040084)-- (-2.0110116053141995,-3.1662923674000636);
			\draw [line width=2pt] (-2.0110116053141995,-3.1662923674000636)-- (-2.733566249156317,-4.351281983301138);
			\draw [line width=2pt] (-4.005262422318443,-3.137390181646379)-- (-3.3405121499836956,-4.322379797547453);
			\draw [line width=2pt] (-4.005262422318443,-3.137390181646379)-- (-4.063066793825813,-4.293477611793768);
			\draw [line width=2pt] (-4,-2)-- (-4.005262422318443,-3.137390181646379);
			\draw [line width=2pt] (-4.005262422318443,-3.137390181646379)-- (-2.0110116053141995,-3.1662923674000636);
			\draw [line width=2pt] (-2,-2)-- (-2.0110116053141995,-3.1662923674000636);
			\draw [line width=2pt] (-4,-2)-- (-2,-2);
			\draw [line width=2pt] (-1.0572394754426049,-3.02178143863164)-- (-1.02833728968892,-4.293477611793768);
			\draw [line width=2pt] (-1.0572394754426049,-3.02178143863164)-- (-0.10346734557100985,-4.293477611793768);
			\draw [line width=2pt] (-1.0861416611962895,-1.9813027514989898)-- (-1.0572394754426049,-3.02178143863164);
			\draw [line width=2pt] (-1.0572394754426049,-3.02178143863164)-- (-0.016760788309955775,-3.02178143863164);
			\draw [line width=2pt] (0,-2)-- (-0.016760788309955775,-3.02178143863164);
			\draw [line width=2pt] (-1.0861416611962895,-1.9813027514989898)-- (0,-2);
			\draw [line width=2pt] (0,-2)-- (1.1682288275911168,-2.617150838080054);
			\draw [line width=2pt] (0,-2)-- (1.2549353848521707,-1.9813027514989898);
			\draw [line width=2pt] (2,-2)-- (2.0063922144479727,-3.1084879958926943);
			\draw [line width=2pt] (2,-2)-- (3.509305873639577,-1.9813027514989898);
			\draw [line width=2pt] (3.509305873639577,-1.9813027514989898)-- (3.509305873639577,-3.0795858101390094);
			\draw [line width=2pt] (2.0063922144479727,-3.1084879958926943)-- (3.509305873639577,-3.0795858101390094);
			\draw [line width=2pt] (2.0063922144479727,-3.1084879958926943)-- (1.4861528708816483,-4.293477611793768);
			\draw [line width=2pt] (2.0063922144479727,-3.1084879958926943)-- (2.7867512297974595,-4.322379797547453);
			\draw [line width=2pt] (3.509305873639577,-3.0795858101390094)-- (4.520882375018541,-3.1084879958926943);
			\draw [line width=2pt] (3.509305873639577,-1.9813027514989898)-- (4.463078003511172,-1.9813027514989898);
			\draw [line width=2pt] (5.243437018860659,-1.9234983799916203)-- (6,-2);
			\draw [line width=2pt] (6,-2)-- (5.994893848456461,-3.137390181646379);
			\draw [line width=2pt] (5.994893848456461,-3.137390181646379)-- (7.237687835864903,-3.137390181646379);
			\draw [line width=2pt] (7.237687835864903,-2.0102049372526745)-- (7.237687835864903,-3.137390181646379);
			\draw [line width=2pt] (6,-2)-- (7.237687835864903,-2.0102049372526745);
			\draw [line width=2pt] (7.237687835864903,-3.137390181646379)-- (7.208785650111218,-4.293477611793768);
			\draw [line width=2pt] (5.994893848456461,-3.137390181646379)-- (5.6191654336585595,-4.293477611793768);
			\draw [line width=2pt] (5.994893848456461,-3.137390181646379)-- (6.370622263254361,-4.322379797547453);
			\draw [line width=2pt] (9.711046713777634,-1.1015000124800458)-- (9.723275810681786,-2.0102049372526745);
			\draw [line width=2pt] (8.9429167953323,-2.0391071230063593)-- (8.926483153660023,-3.1534354774030278);
			\draw [line width=2pt] (8.926483153660023,-3.1534354774030278)-- (9.741222235320619,-3.1534354774030278);
			\draw [line width=2pt] (9.723275810681786,-2.0102049372526745)-- (9.741222235320619,-3.1534354774030278);
			\draw [line width=2pt] (8.9429167953323,-2.0391071230063593)-- (9.723275810681786,-2.0102049372526745);
			\draw [line width=2pt] (8,-2)-- (8.9429167953323,-2.0391071230063593);
			\draw [line width=2pt] (8.926483153660023,-3.1534354774030278)-- (8.956658675203009,-4.330280817579443);
			\draw [line width=2pt] (9.741222235320619,-3.1534354774030278)-- (9.741222235320619,-4.330280817579443);
			\draw [line width=2pt] (-3.92828902365159,10.274671609225438)-- (-3.2342520281629348,9.942740872252601);
			\draw [line width=2pt] (-3.2342520281629348,9.942740872252601)-- (-2.3591619034163696,10.063442958424542);
			\draw [line width=2pt] (-1.9367046018145795,10.757479953913197)-- (-2.3591619034163696,10.063442958424542);
			\draw [line width=2pt] (-1.9668801233575646,11.602394557116778)-- (-1.9367046018145795,10.757479953913197);
			\draw [line width=2pt] (-2.5703905542172647,12.085202901804509)-- (-1.9668801233575646,11.602394557116778);
			\draw [line width=2pt] (-3.44548067896383,12.085202901804509)-- (-2.5703905542172647,12.085202901804509);
			\draw [line width=2pt] (-3.44548067896383,12.085202901804509)-- (-4.079166631366515,11.210112777057972);
			\draw [line width=2pt] (-4.079166631366515,11.210112777057972)-- (-3.92828902365159,10.274671609225438);
			\draw [line width=2pt] (-1.0616144770680143,11.421341427858838)-- (-1.0616144770680143,10.335022652311379);
			\draw [line width=2pt] (-1.0616144770680143,10.335022652311379)-- (0.08505534156541607,9.459932527564842);
			\draw [line width=2pt] (1.2317251601988464,10.304847130768392)-- (0.08505534156541607,9.459932527564842);
			\draw [line width=2pt] (1.3524272463707865,11.300639341686898)-- (1.2317251601988464,10.304847130768392);
			\draw [line width=2pt] (-1.0616144770680143,11.421341427858838)-- (1.2317251601988464,10.304847130768392);
			\draw [line width=2pt] (-3.92828902365159,8.252911665845442)-- (-3.92828902365159,7.498523627270816);
			\draw [line width=2pt] (-3.92828902365159,7.498523627270816)-- (-2.992847855819056,6.774311110239177);
			\draw [line width=2pt] (-2.1479332526154744,7.377821541098876)-- (-2.992847855819056,6.774311110239177);
			\draw [line width=2pt] (-1.9367046018145795,8.252911665845442)-- (-2.1479332526154744,7.377821541098876);
			\draw [line width=2pt] (-3.5360072435927847,9.007299704420067)-- (-3.92828902365159,8.252911665845442);
			\draw [line width=2pt] (-3.5360072435927847,9.007299704420067)-- (-2.449688468045325,9.007299704420067);
			\draw [line width=2pt] (-2.449688468045325,9.007299704420067)-- (-1.9367046018145795,8.252911665845442);
			\draw [line width=2pt] (-2.992847855819056,6.774311110239177)-- (-2.962672334276071,5.778518899320671);
			\draw [line width=2pt] (-0.03564674460652397,8.735720010533203)-- (1.0204965093979517,8.102034058130517);
			\draw [line width=2pt] (1.0204965093979517,8.102034058130517)-- (1.0204965093979517,7.257119454926937);
			\draw [line width=2pt] (1.0204965093979517,7.257119454926937)-- (0.024704298479446066,6.713960067153208);
			\draw [line width=2pt] (-1.1219655201539847,7.226943933383953)-- (0.024704298479446066,6.713960067153208);
			\draw [line width=2pt] (-1.1521410416969697,8.102034058130517)-- (-1.1219655201539847,7.226943933383953);
			\draw [line width=2pt] (-0.03564674460652397,8.735720010533203)-- (-1.1521410416969697,8.102034058130517);
			\draw [line width=2pt] (0.024704298479446066,6.713960067153208)-- (-0.005471223063538952,5.748343377777687);
			\draw [line width=2pt] (-0.03564674460652397,8.735720010533203)-- (0.024704298479446066,6.713960067153208);
			\draw [line width=2pt] (2.9215543666060078,8.735720010533203)-- (2,8);
			\draw [line width=2pt] (2.9215543666060078,8.735720010533203)-- (4,8);
			\draw [line width=2pt] (4,8)-- (4.068224185239439,6.864837674868133);
			\draw [line width=2pt] (4.068224185239439,6.864837674868133)-- (3.012080931234963,6.080274114750522);
			\draw [line width=2pt] (4.068224185239439,6.864837674868133)-- (4.038048663696453,5.6879923346917165);
			\draw [line width=2pt] (2.0162887203164575,6.864837674868133)-- (3.012080931234963,6.080274114750522);
			\draw [line width=2pt] (2,8)-- (2.0162887203164575,6.864837674868133);
			\draw [line width=2pt] (5.99945756399048,6.623433502524252)-- (5.184718482329885,5.657816813148732);
			\draw [line width=2pt] (5.99945756399048,6.623433502524252)-- (6.995249774908985,5.748343377777687);
			\draw [line width=2pt] (2.9215543666060078,8.735720010533203)-- (3.012080931234963,6.080274114750522);
			\draw [line width=2pt] (5.99945756399048,8.554666881275294)-- (7.14612738262391,8.011507493501563);
			\draw [line width=2pt] (7.14612738262391,8.011507493501563)-- (7.14612738262391,6.925188717954103);
			\draw [line width=2pt] (7.14612738262391,6.925188717954103)-- (5.99945756399048,6.623433502524252);
			\draw [line width=2pt] (5.99945756399048,8.554666881275294)-- (5.033840874614959,7.981331971958578);
			\draw [line width=2pt] (5.033840874614959,7.981331971958578)-- (5.094191917700929,6.925188717954103);
			\draw [line width=2pt] (5.094191917700929,6.925188717954103)-- (5.99945756399048,6.623433502524252);
			\draw [line width=2pt] (8.896307632117042,8.645193445904248)-- (8,8);
			\draw [line width=2pt] (8,8)-- (8.021217507370476,7.136417368754997);
			\draw [line width=2pt] (8.021217507370476,7.136417368754997)-- (9.047185239831967,6.442380373266342);
			\draw [line width=2pt] (9.801573278406591,7.136417368754997)-- (9.047185239831967,6.442380373266342);
			\draw [line width=2pt] (8.896307632117042,8.645193445904248)-- (9.741222235320622,8.011507493501563);
			\draw [line width=2pt] (9.741222235320622,8.011507493501563)-- (9.801573278406591,7.136417368754997);
			\draw [line width=2pt] (8.021217507370476,7.136417368754997)-- (8,6);
			\draw [line width=2pt] (9.047185239831967,6.442380373266342)-- (9.017009718288982,5.627641291605746);
			\draw [line width=2pt] (-0.39983089953494266,12.010932677380072)-- (0.7431964376790379,12.010932677380072);
			\draw [line width=2pt] (0.7431964376790379,12.010932677380072)-- (1.3524272463707865,11.300639341686898);
			\begin{scriptsize}
				\draw [fill=rvwvcq] (-2.9809007718018217,4.790154652324867) circle (2.5pt);
				\draw [fill=rvwvcq] (-4,4) circle (2.5pt);
				\draw [fill=rvwvcq] (-3.9461834213678677,3.165654583542984) circle (2.5pt);
				\draw [fill=rvwvcq] (-2.933813813286405,2.6241545606156893) circle (2.5pt);
				\draw [fill=rvwvcq] (-1.9920746429780674,3.2598285005738177) circle (2.5pt);
				\draw [fill=rvwvcq] (-2,4) circle (2.5pt);
				\draw [fill=rvwvcq] (-4,2) circle (2.5pt);
				\draw [fill=rvwvcq] (-2,2) circle (2.5pt);
				\draw [fill=rvwvcq] (-0.03796586458826717,4.743067693809451) circle (2.5pt);
				\draw [fill=rvwvcq] (-1.1209659104428553,4.013219836820489) circle (2.5pt);
				\draw [fill=rvwvcq] (-1.0974224311851468,2.9302197909658996) circle (2.5pt);
				\draw [fill=rvwvcq] (-0.03796586458826717,2.3651762887808965) circle (2.5pt);
				\draw [fill=rvwvcq] (0.8566863472046534,2.906676311708191) circle (2.5pt);
				\draw [fill=rvwvcq] (0.833142867946945,4.036763316078197) circle (2.5pt);
				\draw [fill=rvwvcq] (-1.0974224311851468,2.0591110584306866) circle (2.5pt);
				\draw [fill=rvwvcq] (0.8566863472046534,4.790154652324867) circle (2.5pt);
				\draw [fill=rvwvcq] (2.010316830832367,4.813698131582576) circle (2.5pt);
				\draw [fill=rvwvcq] (3.6583603788719574,4.813698131582576) circle (2.5pt);
				\draw [fill=rvwvcq] (2.010316830832367,3.6365241686971532) circle (2.5pt);
				\draw [fill=rvwvcq] (3.6112734203565404,3.612980689439445) circle (2.5pt);
				\draw [fill=rvwvcq] (2.834338604852162,2.8595893531927743) circle (2.5pt);
				\draw [fill=rvwvcq] (2,2) circle (2.5pt);
				\draw [fill=rvwvcq] (2.8578820841098707,2.0120240999152696) circle (2.5pt);
				\draw [fill=rvwvcq] (3.634816899614249,1.9649371413998526) circle (2.5pt);
				\draw [fill=rvwvcq] (4.553012590664878,4.790154652324867) circle (2.5pt);
				\draw [fill=rvwvcq] (6.248143097219885,4.790154652324867) circle (2.5pt);
				\draw [fill=rvwvcq] (4.576556069922587,3.5894372101817362) circle (2.5pt);
				\draw [fill=rvwvcq] (6.318773534993011,3.5423502516663192) circle (2.5pt);
				\draw [fill=rvwvcq] (5.306403926911548,2.8360458739350656) circle (2.5pt);
				\draw [fill=rvwvcq] (4.576556069922587,2.176828454719229) circle (2.5pt);
				\draw [fill=rvwvcq] (5.3534908854269645,2.1061980169461036) circle (2.5pt);
				\draw [fill=rvwvcq] (6.318773534993011,2.129741496203812) circle (2.5pt);
				\draw [fill=rvwvcq] (7.142795309012806,4.766611173067159) circle (2.5pt);
				\draw [fill=rvwvcq] (8.908556253340938,4.813698131582576) circle (2.5pt);
				\draw [fill=rvwvcq] (7.213425746785932,3.518806772408611) circle (2.5pt);
				\draw [fill=rvwvcq] (8.908556253340938,3.565893730924028) circle (2.5pt);
				\draw [fill=rvwvcq] (8.108077958578852,3.07148066651215) circle (2.5pt);
				\draw [fill=rvwvcq] (7.354686622332182,2.082654537688395) circle (2.5pt);
				\draw [fill=rvwvcq] (8.720208419279272,2.035567579172978) circle (2.5pt);
				\draw [fill=rvwvcq] (9.826751944391567,4.790154652324867) circle (2.5pt);
				\draw [fill=rvwvcq] (-4.040357338398701,1.1644588466377652) circle (2.5pt);
				\draw [fill=rvwvcq] (-2.0627050807511926,1.1880023258954735) circle (2.5pt);
				\draw [fill=rvwvcq] (-4.0639008176564095,0.31689359336026063) circle (2.5pt);
				\draw [fill=rvwvcq] (-2.015618122235776,0.3404370726179691) circle (2.5pt);
				\draw [fill=rvwvcq] (-3.0279877303172387,-0.43649774288641) circle (2.5pt);
				\draw [fill=rvwvcq] (-3.9932703798832847,-1.1428021206176637) circle (2.5pt);
				\draw [fill=rvwvcq] (-3.0279877303172387,-1.1192586413599552) circle (2.5pt);
				\draw [fill=rvwvcq] (-2.015618122235776,-1.1428021206176637) circle (2.5pt);
				\draw [fill=rvwvcq] (-1.1209659104428553,1.211545805153182) circle (2.5pt);
				\draw [fill=rvwvcq] (0.6918819924006944,1.211545805153182) circle (2.5pt);
				\draw [fill=rvwvcq] (-1.168052868958272,0.3404370726179691) circle (2.5pt);
				\draw [fill=rvwvcq] (-0.20277021939222623,-0.3423238258555762) circle (2.5pt);
				\draw [fill=rvwvcq] (0.7154254716584028,0.26980663484484374) circle (2.5pt);
				\draw [fill=rvwvcq] (-1.1445093897005638,-1.1192586413599552) circle (2.5pt);
				\draw [fill=rvwvcq] (-0.20277021939222623,-1.1192586413599552) circle (2.5pt);
				\draw [fill=rvwvcq] (1.5629907249359065,1.211545805153182) circle (2.5pt);
				\draw [fill=rvwvcq] (2.504729895244244,1.1880023258954735) circle (2.5pt);
				\draw [fill=rvwvcq] (4.011512567737584,1.211545805153182) circle (2.5pt);
				\draw [fill=rvwvcq] (2.5282733745019526,-0.012715116247657764) circle (2.5pt);
				\draw [fill=rvwvcq] (4,0) circle (2.5pt);
				\draw [fill=rvwvcq] (2.033860310090075,-1.0957151621022467) circle (2.5pt);
				\draw [fill=rvwvcq] (2.551816853759661,-1.1192586413599552) circle (2.5pt);
				\draw [fill=rvwvcq] (3.3758386277794563,-1.0957151621022467) circle (2.5pt);
				\draw [fill=rvwvcq] (4.035056046995292,-1.1192586413599552) circle (2.5pt);
				\draw [fill=rvwvcq] (4.906164779530505,1.1880023258954735) circle (2.5pt);
				\draw [fill=rvwvcq] (6.412947452023845,1.211545805153182) circle (2.5pt);
				\draw [fill=rvwvcq] (4.929708258788213,0.05791532152546761) circle (2.5pt);
				\draw [fill=rvwvcq] (6.412947452023845,0.05791532152546761) circle (2.5pt);
				\draw [fill=rvwvcq] (5.023882175819047,-1.0957151621022467) circle (2.5pt);
				\draw [fill=rvwvcq] (6.03625178390051,-1.1428021206176637) circle (2.5pt);
				\draw [fill=rvwvcq] (5.518295240230924,-1.1192586413599552) circle (2.5pt);
				\draw [fill=rvwvcq] (6.436490931281553,-1.1428021206176637) circle (2.5pt);
				\draw [fill=rvwvcq] (7.260512705301348,1.1880023258954735) circle (2.5pt);
				\draw [fill=rvwvcq] (8.932099732598648,1.211545805153182) circle (2.5pt);
				\draw [fill=rvwvcq] (7.260512705301348,-0.03625859550536622) circle (2.5pt);
				\draw [fill=rvwvcq] (8.955643211856355,-0.03625859550536622) circle (2.5pt);
				\draw [fill=rvwvcq] (7.354686622332182,-1.1663455998753722) circle (2.5pt);
				\draw [fill=rvwvcq] (8.060991000063435,-1.1898890791330807) circle (2.5pt);
				\draw [fill=rvwvcq] (8.696664940021563,-1.1663455998753722) circle (2.5pt);
				\draw [fill=rvwvcq] (9.826751944391567,1.1880023258954735) circle (2.5pt);
				\draw [fill=rvwvcq] (0.08505534156541607,9.459932527564842) circle (2.5pt);
				\draw [fill=rvwvcq] (1.3524272463707865,11.300639341686898) circle (2.5pt);
				\draw [fill=rvwvcq] (-1.0616144770680143,11.421341427858838) circle (2.5pt);
				\draw [fill=rvwvcq] (-0.39983089953494266,12.010932677380072) circle (2.5pt);
				\draw [fill=rvwvcq] (2,10) circle (2.5pt);
				\draw [fill=rvwvcq] (1.9861131987734724,11.42134142785887) circle (2.5pt);
				\draw [fill=rvwvcq] (3.0120809312349617,12.02485185871854) circle (2.5pt);
				\draw [fill=rvwvcq] (4.007873142153467,11.632570078659734) circle (2.5pt);
				\draw [fill=rvwvcq] (4.219101792954362,10.455724738483347) circle (2.5pt);
				\draw [fill=rvwvcq] (3.042256452777948,9.5504590921938) circle (2.5pt);
				\draw [fill=rvwvcq] (5.064016396157944,11.270463820143945) circle (2.5pt);
				\draw [fill=rvwvcq] (5.033840874614957,9.490108049107826) circle (2.5pt);
				\draw [fill=rvwvcq] (6.452090387135253,11.270463820143913) circle (2.5pt);
				\draw [fill=rvwvcq] (6.42191486559227,9.520283570650813) circle (2.5pt);
				\draw [fill=rvwvcq] (5.727877870103613,12.055027380261524) circle (2.5pt);
				\draw [fill=rvwvcq] (7.206478425709878,11.240288298600928) circle (2.5pt);
				\draw [fill=rvwvcq] (7.236653947252865,9.520283570650813) circle (2.5pt);
				\draw [fill=rvwvcq] (8,12) circle (2.5pt);
				\draw [fill=rvwvcq] (8.021217507370473,10.697128910827228) circle (2.5pt);
				\draw [fill=rvwvcq] (8.081568550456444,9.459932527564842) circle (2.5pt);
				\draw [fill=rvwvcq] (9.771397756863607,9.490108049107828) circle (2.5pt);
				\draw [fill=rvwvcq] (9.771397756863603,10.666953389284242) circle (2.5pt);
				\draw [fill=rvwvcq] (9.711046713777634,11.964500815632569) circle (2.5pt);
				\draw [fill=rvwvcq] (-4,-2) circle (2.5pt);
				\draw [fill=rvwvcq] (-2,-2) circle (2.5pt);
				\draw [fill=rvwvcq] (-4.005262422318443,-3.137390181646379) circle (2.5pt);
				\draw [fill=rvwvcq] (-2.0110116053141995,-3.1662923674000636) circle (2.5pt);
				\draw [fill=rvwvcq] (-4.063066793825813,-4.293477611793768) circle (2.5pt);
				\draw [fill=rvwvcq] (-3.3405121499836956,-4.322379797547453) circle (2.5pt);
				\draw [fill=rvwvcq] (-2.733566249156317,-4.351281983301138) circle (2.5pt);
				\draw [fill=rvwvcq] (-2.0110116053141995,-4.264575426040084) circle (2.5pt);
				\draw [fill=rvwvcq] (-1.0861416611962895,-1.9813027514989898) circle (2.5pt);
				\draw [fill=rvwvcq] (-1.0572394754426049,-3.02178143863164) circle (2.5pt);
				\draw [fill=rvwvcq] (-0.016760788309955775,-3.02178143863164) circle (2.5pt);
				\draw [fill=rvwvcq] (0,-2) circle (2.5pt);
				\draw [fill=rvwvcq] (1.2549353848521707,-1.9813027514989898) circle (2.5pt);
				\draw [fill=rvwvcq] (1.1682288275911168,-2.617150838080054) circle (2.5pt);
				\draw [fill=rvwvcq] (-1.02833728968892,-4.293477611793768) circle (2.5pt);
				\draw [fill=rvwvcq] (-0.10346734557100985,-4.293477611793768) circle (2.5pt);
				\draw [fill=rvwvcq] (2,-2) circle (2.5pt);
				\draw [fill=rvwvcq] (2.0063922144479727,-3.1084879958926943) circle (2.5pt);
				\draw [fill=rvwvcq] (3.509305873639577,-1.9813027514989898) circle (2.5pt);
				\draw [fill=rvwvcq] (3.509305873639577,-3.0795858101390094) circle (2.5pt);
				\draw [fill=rvwvcq] (4.463078003511172,-1.9813027514989898) circle (2.5pt);
				\draw [fill=rvwvcq] (4.520882375018541,-3.1084879958926943) circle (2.5pt);
				\draw [fill=rvwvcq] (1.4861528708816483,-4.293477611793768) circle (2.5pt);
				\draw [fill=rvwvcq] (2.7867512297974595,-4.322379797547453) circle (2.5pt);
				\draw [fill=rvwvcq] (5.243437018860659,-1.9234983799916203) circle (2.5pt);
				\draw [fill=rvwvcq] (6,-2) circle (2.5pt);
				\draw [fill=rvwvcq] (7.237687835864903,-2.0102049372526745) circle (2.5pt);
				\draw [fill=rvwvcq] (5.994893848456461,-3.137390181646379) circle (2.5pt);
				\draw [fill=rvwvcq] (7.237687835864903,-3.137390181646379) circle (2.5pt);
				\draw [fill=rvwvcq] (7.208785650111218,-4.293477611793768) circle (2.5pt);
				\draw [fill=rvwvcq] (6.370622263254361,-4.322379797547453) circle (2.5pt);
				\draw [fill=rvwvcq] (5.6191654336585595,-4.293477611793768) circle (2.5pt);
				\draw [fill=rvwvcq] (9.723275810681786,-2.0102049372526745) circle (2.5pt);
				\draw [fill=rvwvcq] (9.741222235320619,-3.1534354774030278) circle (2.5pt);
				\draw [fill=rvwvcq] (8.926483153660023,-3.1534354774030278) circle (2.5pt);
				\draw [fill=rvwvcq] (8.9429167953323,-2.0391071230063593) circle (2.5pt);
				\draw [fill=rvwvcq] (8.956658675203009,-4.330280817579443) circle (2.5pt);
				\draw [fill=rvwvcq] (9.741222235320619,-4.330280817579443) circle (2.5pt);
				\draw [fill=rvwvcq] (8,-2) circle (2.5pt);
				\draw [fill=rvwvcq] (9.711046713777634,-1.1015000124800458) circle (2.5pt);
				\draw [fill=rvwvcq] (1.2317251601988464,10.304847130768392) circle (2.5pt);
				\draw [fill=rvwvcq] (-1.0616144770680143,10.335022652311379) circle (2.5pt);
				\draw [fill=rvwvcq] (-3.92828902365159,10.274671609225438) circle (2.5pt);
				\draw [fill=rvwvcq] (-3.2342520281629348,9.942740872252601) circle (2.5pt);
				\draw [fill=rvwvcq] (-2.3591619034163696,10.063442958424542) circle (2.5pt);
				\draw [fill=rvwvcq] (-1.9367046018145795,10.757479953913197) circle (2.5pt);
				\draw [fill=rvwvcq] (-1.9668801233575646,11.602394557116778) circle (2.5pt);
				\draw [fill=rvwvcq] (-2.5703905542172647,12.085202901804509) circle (2.5pt);
				\draw [fill=rvwvcq] (-3.44548067896383,12.085202901804509) circle (2.5pt);
				\draw [fill=rvwvcq] (-4.079166631366515,11.210112777057972) circle (2.5pt);
				\draw [fill=rvwvcq] (-3.5360072435927847,9.007299704420067) circle (2.5pt);
				\draw [fill=rvwvcq] (-2.449688468045325,9.007299704420067) circle (2.5pt);
				\draw [fill=rvwvcq] (-3.92828902365159,8.252911665845442) circle (2.5pt);
				\draw [fill=rvwvcq] (-3.92828902365159,7.498523627270816) circle (2.5pt);
				\draw [fill=rvwvcq] (-1.9367046018145795,8.252911665845442) circle (2.5pt);
				\draw [fill=rvwvcq] (-2.1479332526154744,7.377821541098876) circle (2.5pt);
				\draw [fill=rvwvcq] (-2.992847855819056,6.774311110239177) circle (2.5pt);
				\draw [fill=rvwvcq] (-2.962672334276071,5.778518899320671) circle (2.5pt);
				\draw [fill=rvwvcq] (-0.03564674460652397,8.735720010533203) circle (2.5pt);
				\draw [fill=rvwvcq] (-1.1521410416969697,8.102034058130517) circle (2.5pt);
				\draw [fill=rvwvcq] (-1.1219655201539847,7.226943933383953) circle (2.5pt);
				\draw [fill=rvwvcq] (0.024704298479446066,6.713960067153208) circle (2.5pt);
				\draw [fill=rvwvcq] (-0.005471223063538952,5.748343377777687) circle (2.5pt);
				\draw [fill=rvwvcq] (1.0204965093979517,8.102034058130517) circle (2.5pt);
				\draw [fill=rvwvcq] (1.0204965093979517,7.257119454926937) circle (2.5pt);
				\draw [fill=rvwvcq] (2.9215543666060078,8.735720010533203) circle (2.5pt);
				\draw [fill=rvwvcq] (2,8) circle (2.5pt);
				\draw [fill=rvwvcq] (4,8) circle (2.5pt);
				\draw [fill=rvwvcq] (2.0162887203164575,6.864837674868133) circle (2.5pt);
				\draw [fill=rvwvcq] (4.068224185239439,6.864837674868133) circle (2.5pt);
				\draw [fill=rvwvcq] (3.012080931234963,6.080274114750522) circle (2.5pt);
				\draw [fill=rvwvcq] (4.038048663696453,5.6879923346917165) circle (2.5pt);
				\draw [fill=rvwvcq] (5.99945756399048,6.623433502524252) circle (2.5pt);
				\draw [fill=rvwvcq] (5.094191917700929,6.925188717954103) circle (2.5pt);
				\draw [fill=rvwvcq] (5.033840874614959,7.981331971958578) circle (2.5pt);
				\draw [fill=rvwvcq] (5.99945756399048,8.554666881275294) circle (2.5pt);
				\draw [fill=rvwvcq] (7.14612738262391,8.011507493501563) circle (2.5pt);
				\draw [fill=rvwvcq] (7.14612738262391,6.925188717954103) circle (2.5pt);
				\draw [fill=rvwvcq] (6.995249774908985,5.748343377777687) circle (2.5pt);
				\draw [fill=rvwvcq] (5.184718482329885,5.657816813148732) circle (2.5pt);
				\draw [fill=rvwvcq] (8.896307632117042,8.645193445904248) circle (2.5pt);
				\draw [fill=rvwvcq] (8,8) circle (2.5pt);
				\draw [fill=rvwvcq] (8.021217507370476,7.136417368754997) circle (2.5pt);
				\draw [fill=rvwvcq] (8,6) circle (2.5pt);
				\draw [fill=rvwvcq] (9.047185239831967,6.442380373266342) circle (2.5pt);
				\draw [fill=rvwvcq] (9.017009718288982,5.627641291605746) circle (2.5pt);
				\draw [fill=rvwvcq] (9.801573278406591,7.136417368754997) circle (2.5pt);
				\draw [fill=rvwvcq] (9.741222235320622,8.011507493501563) circle (2.5pt);
				\draw [fill=rvwvcq] (-0.018821787130282503,10.42339470902732) circle (2.5pt);
				\draw [fill=rvwvcq] (0.7431964376790379,12.010932677380072) circle (2.5pt);
				\draw [fill=rvwvcq] (2.4577374435000086,11.121911415102531) circle (2.5pt);
				\draw [fill=rvwvcq] (3.6007647807139893,10.61389926522965) circle (2.5pt);
				\draw [fill=rvwvcq] (5.63281338020551,10.550397746495541) circle (2.5pt);
				\draw [fill=rvwvcq] (8.680886279442792,10.010634837255605) circle (2.5pt);
				\draw [fill=rvwvcq] (9.030144632480397,11.375917490038972) circle (2.5pt);
				\draw [fill=rvwvcq] (-0.4950831776361077,7.724580162827642) circle (2.5pt);
				\draw [fill=rvwvcq] (2.584740480968229,7.407072569157092) circle (2.5pt);
			\end{scriptsize}
		\end{tikzpicture}
	
			\caption{Twenty five possibilities for the complement of $G_3$ where \(I_{G_3}(z)= 1+8z +8z^2\). }
			\label{M_1}
		\end{figure}  
	\FloatBarrier
\end{proof}
It is in fact possible to construct more graphs with independence polynomials given in Lemma~\ref{examples-disconnected}.
\par 
The proof of Theorem~\ref{exampls-ind-no-4} is now presented.
\begin{proof}[Proof of Theorem~\ref{exampls-ind-no-4}]
The reduced independence polynomials of the graphs $G'$ and $G''$ whose complements are given in Figures~(\ref{Com G1}) and ~(\ref{Com G2}), are $16z+ 20z^2 +8z^3+z^4 $ and $16z+ 40z^2 +32 z^3+ 8z^4 $  respectively. It follows from Remark~\ref{indpoly} that $\mathcal{A}(G')= [-4,0]$ and $\mathcal{A}(G'')= [-2,0]$.
		\begin{figure}[h!]
		
		\begin{center}
			\begin{multicols}{2}
				\definecolor{ududff}{rgb}{0.30196078431372547,0.30196078431372547,1}
				\begin{tikzpicture}
					
					\draw [line width=2pt] (-9.04210631229236,-0.4900354374307869)-- (-8,-1);
					\draw [line width=2pt] (-8,-1)-- (-8.994939091915835,-1.6927995570321166);
					\draw [line width=2pt] (-9.04210631229236,-0.4900354374307869)-- (-8.994939091915835,-1.6927995570321166);
					\draw [line width=2pt] (-8,-1)-- (-7,-1);
					\draw [line width=2pt] (-6.046987818383169,-0.4900354374307869)-- (-7,-1);
					\draw [line width=2pt] (-7,-1)-- (-6.046987818383169,-1.6927995570321166);
					\draw [line width=2pt] (-6.046987818383169,-0.4900354374307869)-- (-6.046987818383169,-1.6927995570321166);
					\draw [line width=2pt] (-9.04210631229236,-2.6125603543743097)-- (-8.00442746400886,-3.4379867109634574);
					\draw [line width=2pt] (-8.00442746400886,-3.4379867109634574)-- (-9.04210631229236,-4.074744186046514);
					\draw [line width=2pt] (-9.04210631229236,-2.6125603543743097)-- (-9.04210631229236,-4.074744186046514);
					\draw [line width=2pt] (-8.00442746400886,-3.4379867109634574)-- (-6.919581395348838,-3.414403100775196);
					\draw [line width=2pt] (-6.046987818383169,-2.6125603543743097)-- (-6.919581395348838,-3.414403100775196);
					\draw [line width=2pt] (-6.919581395348838,-3.414403100775196)-- (-5.999820598006647,-4.074744186046514);
					\draw [line width=2pt] (-6.046987818383169,-2.6125603543743097)-- (-5.999820598006647,-4.074744186046514);
					\draw [line width=2pt] (-9.06568992248062,-5.301091915836105)-- (-9.06568992248062,-7.117029900332231);
					\draw [line width=2pt] (-9.06568992248062,-7.117029900332231)-- (-5.952653377630123,-7.069862679955707);
					\draw [line width=2pt] (-5.999820598006647,-5.253924695459583)-- (-5.952653377630123,-7.069862679955707);
					\draw [line width=2pt] (-9.06568992248062,-5.301091915836105)-- (-5.999820598006647,-5.253924695459583);
					\draw [line width=2pt] (-9.06568992248062,-5.301091915836105)-- (-5.952653377630123,-7.069862679955707);
					\draw [line width=2pt] (-5.999820598006647,-5.253924695459583)-- (-9.06568992248062,-7.117029900332231);
					\begin{scriptsize}
						\draw [fill=ududff] (-9.04210631229236,-0.4900354374307869) circle (2.5pt);
						\draw [fill=ududff] (-8.994939091915835,-1.6927995570321166) circle (2.5pt);
						\draw [fill=ududff] (-9.04210631229236,-2.6125603543743097) circle (2.5pt);
						\draw [fill=ududff] (-8,-1) circle (2.5pt);
						\draw [fill=ududff] (-7,-1) circle (2.5pt);
						\draw [fill=ududff] (-6.046987818383169,-0.4900354374307869) circle (2.5pt);
						\draw [fill=ududff] (-6.046987818383169,-1.6927995570321166) circle (2.5pt);
						\draw [fill=ududff] (-6.046987818383169,-2.6125603543743097) circle (2.5pt);
						\draw [fill=ududff] (-5.999820598006647,-4.074744186046514) circle (2.5pt);
						\draw [fill=ududff] (-9.04210631229236,-4.074744186046514) circle (2.5pt);
						\draw [fill=ududff] (-8.00442746400886,-3.4379867109634574) circle (2.5pt);
						\draw [fill=ududff] (-6.919581395348838,-3.414403100775196) circle (2.5pt);
						\draw [fill=ududff] (-9.06568992248062,-5.301091915836105) circle (2.5pt);
						\draw [fill=ududff] (-5.999820598006647,-5.253924695459583) circle (2.5pt);
						\draw [fill=ududff] (-9.06568992248062,-7.117029900332231) circle (2.5pt);
						\draw [fill=ududff] (-5.952653377630123,-7.069862679955707) circle (2.5pt);
					\end{scriptsize}
				\end{tikzpicture}
				\caption{\label{Com G1} Complement of $G'$}
				%\caption{Complement of $G_1$}
				
				\definecolor{ududff}{rgb}{0.30196078431372547,0.30196078431372547,1}
				\begin{tikzpicture}
					\draw [line width=2pt] (-2.652145192487453,-2.4301542052633107)-- (-2.074020416418786,2.8050868224696193);
					\draw [line width=2pt] (-2.652145192487453,-2.4301542052633107)-- (1.2020199813036607,-1.3060226962409023);
					\draw [line width=2pt] (1.2020199813036607,-1.3060226962409023)-- (-2.074020416418786,2.8050868224696193);
					\draw [line width=2pt] (-4.868290167417344,0.9743583649188402)-- (-5.253706684796455,-0.5673077045976054);
					\draw [line width=2pt] (-4.868290167417344,0.9743583649188402)-- (-3.9047488739695653,2.32331617574573);
					\draw [line width=2pt] (0.8808395501544013,1.3597748822979516)-- (-0.5323543469023405,2.5481424775502117);
					\draw [line width=2pt] (0.8808395501544013,1.3597748822979516)-- (1.3304921537633645,0.23564337327554333);
					\draw [line width=2pt] (-0.5323543469023405,2.5481424775502117)-- (-2.652145192487453,-2.4301542052633107);
					\draw [line width=2pt] (0.8808395501544013,1.3597748822979516)-- (-2.652145192487453,-2.4301542052633107);
					\draw [line width=2pt] (1.3304921537633645,0.23564337327554333)-- (-2.652145192487453,-2.4301542052633107);
					\draw [line width=2pt] (-5.253706684796455,-0.5673077045976054)-- (-2.652145192487453,-2.4301542052633107);
					\draw [line width=2pt] (-4.868290167417344,0.9743583649188402)-- (-2.652145192487453,-2.4301542052633107);
					\draw [line width=2pt] (-3.9047488739695653,2.32331617574573)-- (-2.652145192487453,-2.4301542052633107);
					\draw [line width=2pt] (1.3304921537633645,0.23564337327554333)-- (-2.074020416418786,2.8050868224696193);
					\draw [line width=2pt] (0.8808395501544013,1.3597748822979516)-- (-2.074020416418786,2.8050868224696193);
					\draw [line width=2pt] (-0.5323543469023405,2.5481424775502117)-- (-2.074020416418786,2.8050868224696193);
					\draw [line width=2pt] (-3.9047488739695653,2.32331617574573)-- (-2.074020416418786,2.8050868224696193);
					\draw [line width=2pt] (-5.253706684796455,-0.5673077045976054)-- (-2.074020416418786,2.8050868224696193);
					\draw [line width=2pt] (-4.868290167417344,0.9743583649188402)-- (-2.074020416418786,2.8050868224696193);
					\draw [line width=2pt] (1.3304921537633645,0.23564337327554333)-- (1.2020199813036607,-1.3060226962409023);
					\draw [line width=2pt] (-0.5323543469023405,2.5481424775502117)-- (1.2020199813036607,-1.3060226962409023);
					\draw [line width=2pt] (-5.253706684796455,-0.5673077045976054)-- (1.2020199813036607,-1.3060226962409023);
					\draw [line width=2pt] (-3.9047488739695653,2.32331617574573)-- (1.2020199813036607,-1.3060226962409023);
					\draw [line width=2pt] (-4.418637563808381,-1.9805016016543473)-- (-2.652145192487453,-2.4301542052633107);
					\draw [line width=2pt] (-4.418637563808381,-1.9805016016543473)-- (-5.253706684796455,-0.5673077045976054);
					\draw [line width=2pt] (-3.9047488739695653,2.32331617574573)-- (-4.418637563808381,-1.9805016016543473);
					\draw [line width=2pt] (-0.5323543469023405,2.5481424775502117)-- (-4.418637563808381,-1.9805016016543473);
					\draw [line width=2pt] (1.3304921537633645,0.23564337327554333)-- (-4.418637563808381,-1.9805016016543473);
					\draw [line width=2pt] (-0.0184656570635253,-2.141091817228977)-- (-2.652145192487453,-2.4301542052633107);
					\draw [line width=2pt] (-5.253706684796455,-0.5673077045976054)-- (-0.0184656570635253,-2.141091817228977);
					\draw [line width=2pt] (-3.9047488739695653,2.32331617574573)-- (-0.0184656570635253,-2.141091817228977);
					\draw [line width=2pt] (1.3304921537633645,0.23564337327554333)-- (-0.0184656570635253,-2.141091817228977);
					\draw [line width=2pt] (-0.5323543469023405,2.5481424775502117)-- (-0.0184656570635253,-2.141091817228977);
					\draw [line width=2pt] (-4.418637563808381,-1.9805016016543473)-- (-3.55145039970538,-2.9761609382170517);
					\draw [line width=2pt] (-0.0184656570635253,-2.141091817228977)-- (-1.1425971660859335,-2.879806808872274);
					\draw [line width=2pt] (-3.615686485935232,-3.843348102320052)-- (-3.55145039970538,-2.9761609382170517);
					\draw [line width=2pt] (-1.1425971660859335,-2.879806808872274)-- (-1.1104791229710076,-3.7791120160902003);
					\draw [line width=2pt] (-2.620027149372527,-3.8112300592051263)-- (-3.55145039970538,-2.9761609382170517);
					\draw [line width=2pt] (-1.1425971660859335,-2.879806808872274)-- (-2.620027149372527,-3.8112300592051263);
					\draw [line width=2pt] (-3.55145039970538,-2.9761609382170517)-- (-1.1425971660859335,-2.879806808872274);
					\begin{scriptsize}
						\draw [fill=ududff] (-2.074020416418786,2.8050868224696193) circle (2.5pt);
						\draw [fill=ududff] (-2.652145192487453,-2.4301542052633107) circle (2.5pt);
						\draw [fill=ududff] (-0.5323543469023405,2.5481424775502117) circle (2.5pt);
						\draw [fill=ududff] (1.3304921537633645,0.23564337327554333) circle (2.5pt);
						\draw [fill=ududff] (0.8808395501544013,1.3597748822979516) circle (2.5pt);
						\draw [fill=ududff] (1.2020199813036607,-1.3060226962409023) circle (2.5pt);
						\draw [fill=ududff] (-3.9047488739695653,2.32331617574573) circle (2.5pt);
						\draw [fill=ududff] (-4.868290167417344,0.9743583649188402) circle (2.5pt);
						\draw [fill=ududff] (-5.253706684796455,-0.5673077045976054) circle (2.5pt);
						\draw [fill=ududff] (-4.418637563808381,-1.9805016016543473) circle (2.5pt);
						\draw [fill=ududff] (-0.0184656570635253,-2.141091817228977) circle (2.5pt);
						\draw [fill=ududff] (-3.55145039970538,-2.9761609382170517) circle (2.5pt);
						\draw [fill=ududff] (-1.1425971660859335,-2.879806808872274) circle (2.5pt);
						\draw [fill=ududff] (-3.615686485935232,-3.843348102320052) circle (2.5pt);
						\draw [fill=ududff] (-1.1104791229710076,-3.7791120160902003) circle (2.5pt);
						\draw [fill=ududff] (-2.620027149372527,-3.8112300592051263) circle (2.5pt);
					\end{scriptsize}
				\end{tikzpicture}
				\caption{\label{Com G2} Complement of $G''$}	
				
			\end{multicols}	
			
		\end{center}
		
	\end{figure}
It follows from Lemma~\ref{examples-disconnected} that there are at least ten graphs whose independence polynomial is $1+12z+9z^2$. Since the independence polynomial of the path graph on $4$ vertices is $ 1+4z+3z^2$, it follows from  Proposition~\ref{disconnectedgraphs}(2) and  Table~\ref{2components-equalindnumb-k=3}  that there are at least ten non-isomorphic disconnected graphs  with two components, and their independence polynomial is \( (1+4z+3z^2)(1+12z+9z^2) \). The independence attractor in this case is $[-\frac{4}{3},0]$ as the situation corresponds to $k=3$ (see Table ~\ref{2components-equalindnumb-k=3}).
\par
There are at least $25$ non-isomorphic graphs whose independence polynomial is $1+8z +8z^2$ by Lemma~\ref{examples-disconnected}. The number of  graphs with two identical components, each of which is an above-mentioned graph,  is $25$ whereas those with non-isomorphic components is $300$. Thus, there are at least \(325\) disconnected graphs  with two components such that their independence polynomial is \(  (1+8z+8z^2)(1+8z+8z^2) \). It follows from  Proposition~\ref{disconnectedgraphs}(2) and  Table~\ref{2components-equalindnumb-k=4} that this corresponds to $k=4$ and the independence attractor   is $[-1,0]$.
\par
That there is no disconnected graph whose independence attractor is $[-4,0]$ or $[-2,0]$ follows from the last part of Proposition~\ref{disconnectedgraphs}.
\end{proof}
All the graphs considered in the proof of Theorem~\ref{exampls-ind-no-4} are either connected or with exactly two components. However, there are disconnected graphs with three components whose independence attractor is $[-\frac{4}{3},0]$.
    \begin{rem} The number of disconnected graphs \(G\) with three components such that \(I_G(z) = (1+z)(1+3z)(1+12z+9z^2)\) is at least ten (see Lemma~\ref{examples-disconnected} and Table~\ref{3components-k=3}). By Proposition~\ref{disconnectedgraphs}, their independence attractor is $[-\frac{4}{3},0]$.
 
      \end{rem}
      We conclude with three remarks.
      \begin{rem}
      	\begin{enumerate}
      		\item Using the arguments used in the proof of Proposition~\ref{disconnectedgraphs}, it seems possible to prove that, for each $n>4$ there is a graph whose independence attractor is a line segment.  
      		\item Studying graphs whose independence attractors are Jordan curves, other than circles seems to be an interesting problem.
      		\item  What are all the reduced independence polynomials that are conformally conjugate to a unicritical polynomial $z^n + c$ for some $c$? This is an important question as its answer  would lead to different tractable possibilities of independence attractors. The value of $c$, in this case should be highly restricted, since the coefficients of the reduced independence polynomials are all positive integers satisfying certain conditions.
      				\end{enumerate} 
      \end{rem}
 \textbf{Acknowledgement:} Moumita Manna is supported by  University Grants Commission, Govt. of India through a Junior Research Fellowship.
 
\bibliographystyle{amsplain}
\addcontentsline{toc}{chapter}{\numberline{}References}

\end{document}